\documentclass[12pt]{article}
\usepackage{amssymb,amsmath,latexsym,lipsum}
\usepackage{enumitem}
\usepackage{mleftright}

\newlist{inlineroman}{enumerate*}{1}
\setlist[inlineroman]{itemjoin*={{, and }},afterlabel=~,label=\roman*.}

\newcommand{\inlinerom}[1]{
\begin{inlineroman}
#1
\end{inlineroman}
}

\newlist{Inlineroman}{enumerate*}{1}
\setlist[Inlineroman]{itemjoin*={{, and }},afterlabel=~,label=\Roman*.}

\usepackage{amsthm}
\usepackage{euscript}
\usepackage{color}
\usepackage[T1]{fontenc}
\usepackage[utf8]{inputenc}
\usepackage{enumitem}
\usepackage{authblk}
\usepackage[singlespacing]{setspace}
\usepackage[colorinlistoftodos]{todonotes}
\usepackage {tikz}
\usetikzlibrary {positioning}
\definecolor {processblue}{cmyk}{0.96,0,0,0}
\usepackage{bookmark}
\usepackage{multirow,bigdelim}
\usepackage{stackengine}
\usepackage{mathrsfs}
\usepackage{mathtools}

\newtheorem{theorem}{Theorem}[section]
\newtheorem{proposition}[theorem]{Proposition}
\newtheorem{lemma}[theorem]{Lemma}
\newtheorem{corollary}[theorem]{Corollary}

\theoremstyle{remark}
\newtheorem{example}[theorem]{Example}

\theoremstyle{remark}
\newtheorem{remark}[theorem]{Remark}

\newcommand{\Addresses}{{
  \bigskip
  \footnotesize

  Huy Dang, \textsc{Department of Mathematics, University of Virginia,
    Charlottesville, Virginia 22903}\par\nopagebreak
  \textit{E-mail address}: \texttt{hqd4bz@virginia.edu}

}}

\DeclareMathOperator{\spec}{Spec}
\DeclareMathOperator{\Frac}{Frac}

\begin{document}

\title{Connectedness of The Moduli Space of Artin-Schreier Curves of Fixed Genus \vspace{-1ex}}
\author{Huy Dang \vspace{-5ex}}
\date{}
\maketitle

\abstract{We study the moduli space $\mathcal{AS}_{g}$
 of Artin-Schreier curves of genus $g$ over an algebraically closed field $k$ of positive characteristic $p$. The moduli space is partitioned into strata, which are irreducible. Each stratum parameterizes Artin-Schreier curves whose ramification divisors have the same coefficients. We construct deformations of these curves to study the relations between those strata. As an application, when $p=3$, we prove that  $\mathcal{AS}_{g}$ is connected for all $g$. When $p>3$, it turns out that $\mathcal{AS}_{g}$ is connected for a sufficiently large value of $g$. In the course of our work, we answer a question of Pries and Zhu about how a combinatorial graph determines the geometry of $\mathcal{AS}_g$.}
 
\section{Introduction}

Let $k$ be an algebraically closed field of characteristic $p>0$. An Artin-Schreier $k$-curve is a smooth projective connected $k$-curve $Y$ which is a $\mathbb{Z}/p$-cover of the projective line $\mathbb{P}^1_k$. In \cite{MR2985514}, the authors introduce the moduli space $\mathcal{AS}_g$ of Artin-Schreier $k$-curves of genus $g$. They enumerate a family of irreducible strata of $\mathcal{AS}_g$ by partitions of the integer $d+2$ where $d:=2g/(p-1)$ such that no entries are congruent to $1$ modulo $p$. From there, Pries and Zhu give the formula to calculate the dimension of a stratum from its corresponding partition (see \cite[Corollary 3.11]{MR2985514}). Furthermore, they construct a deformation to show whether a stratum lies in the closure of another under certain conditions (see \cite[Proposition 4.6]{MR2985514}). That technique, combine with the results of Maugeais in \cite{MR2223481} which showed that each irreducible component of $\mathcal{AS}_g$ has dimension $2g/(p-1)-1$, allowed them to determine all the cases when the moduli space of such covers is irreducible (see \cite[Corollary 1.2]{MR2985514}).

In this paper, based on the framework of \cite{MR2985514}, we continue to study the relations between these irreducible strata of $\mathcal{AS}_g$ with the aim of finding sufficient conditions for the moduli space to be connected. Corollary 1.2 of \cite{MR2985514} implies that $\mathcal{AS}_g$ is irreducible and thus connected when $p=2$. For $p>2$, the same corollary only tells us that $\mathcal{AS}_g$ is connected for small values of $g$. We are able to prove that for $p=3$, the moduli space $\mathcal{AS}_g$ is also connected for all $g$. When $p>3$, the following result, which implies that $\mathcal{AS}_g$ is connected for $g$ sufficiently large, is the main result of this paper.

\begin{theorem}
\label{connectedness}

When $p=3$, the moduli space $\mathcal{AS}_g$ is connected for any $g$.

When $p=5$,  $\mathcal{AS}_g$  is connected for any $g \ge 14$.

 When $p>5$, $\mathcal{AS}_g$ is connected if $g \ge \frac{(p^3-2p^2+p-8)(p-1)}{8}$.
 
 \end{theorem}

\begin{remark}
\label{char5}
When $p=5$, the only cases that the theorem does not cover are $g=10$ and $g=12$. The techniques we use in this paper are inadequate to study these cases. A more conceptual approach, which will be introduced in my thesis, will show that $\mathcal{AS}_g$ is disconnected when $g=10$ or $g=12$.
\end{remark}

\noindent To prove Theorem \ref{connectedness}, we use an idea of Pries and Zhu in \cite{MR2985514} about how the deformations of Artin-Schreier covers determine the geometry of the moduli space as follows. Suppose $\overrightarrow{E}_1$ and $\overrightarrow{E}_2$ are two partitions that satisfy the conditions in the previous paragraph. Then the stratum indexed by $\overrightarrow{E}_1$, which is denoted by $\Gamma_{\overrightarrow{E}_1}$, is in the closure of $\Gamma_{\overrightarrow{E}_2}$ if and only if each $k$-point of $\Gamma_{\overrightarrow{E}_1}$ can be deformed over $k[[t]]$ to a $k((t))$-point of $\Gamma_{\overrightarrow{E}_2}$ (see Proposition \ref{propreduce}). These kind deformations, which actually change the number of branch points but do not change the genus, are only possible for wildly ramified cover. They were first studied by M\'{e}zard in her thesis \cite{MEZAR1998} and were also used by Pries and Zhu in \cite{MR2985514}. With the aim to prove Theorem \ref{connectedness}, we will construct explicitly some deformations of this type in Theorem \ref{deformationtheorem}.

\begin{remark}
In \cite{MR2985514}, the authors define a combinatorial invariant for $\mathcal{AS}_g$. They call it the graph $G_d$. Each vertex of $G_d$ indexes a stratum of $\mathcal{AS}_g$ of the same type as ones were mentioned previously. More details will be given in the following chapter. The existence of a chain of oriented edges connecting two vertices in $G_d$ is necessary for the stratum corresponding to one vertex to lie in the closure of another (see \cite[Lemma 4.3]{MR2985514}). When considering two vertices connected by a single edge, we able to show whether the necessary condition is also sufficient by comparing the dimensions of the corresponding strata (see Corollary \ref{openquestion1}). That answers Open Question $1$ of \cite{MR2985514}. 
\end{remark}

 Here is an outline of this paper. In Section \ref{Overview}, we give a review on Artin-Schreier theory and some results of \cite{MR2985514} about the relationship between irreducible strata of $\mathcal{AS}_g$ and partitions of integers. Section \ref{relation} introduces some deformation results and gives us information on how irreducible strata fit together in $\mathcal{AS}_g$ by looking at their corresponding partitions (Theorem \ref{deformationtheorem}). Open Question $1$ of \cite{MR2985514} is also answered in this section. Connectedness of $\mathcal{AS}_g$ is discussed in Section \ref{connectednesschap}. In particular, Theorem \ref{connectedness} will be proved in this section using Theorem \ref{deformationtheorem} and graph theory on the "tree" formed by partitions of $d$. The details of how the deformations are constructed are postponed to Section \ref{technical}, which is the most novel part, but also the most technical part of this paper.

 \renewcommand{\abstractname}{Acknowledgements}
 
 \begin{abstract}
    I would like to thank Andrew Obus,  my advisor, for his guidance, support, and patience during the development of this paper. He also gave me the idea for the deformation in \S \ref{deformation4}  (Proposition \ref{proptype4}), which is the last piece of the puzzle to prove the connectedness result. I thank Rachel Pries and David Harbater for useful suggestions and proofreading an earlier version of this paper. I would also thank Jonathan Gerhard and Bennett Rennier for making the programs that help me to check my results. 
\end{abstract}
 
\section{Moduli space of Artin-Schreier curves}
\label{Overview}
\subsection{Artin-Schreier Theory}

\label{ArtinSchreiertheory}

Let $K$ be a field of characteristic $p$. If $L$ is a separable extension of $K$ of degree $p$, then Artin-Schreier theory says that $L=K(\alpha)$ where $\alpha$ is a root of the polynomial $y^p-y=a$ for some element $a \in K$ (detailed in \cite{MR1878556}).

Let $Y$ be an Artin-Schreier $k$-curve. Then there is a $\mathbb{Z}/p$-cover $\phi$: $Y \rightarrow \mathbb{P}_k^1$ with an affine equation of the form $y^p-y=f(x)$ for some non-constant rational function $f(x)\in k(x)$. A cover $\phi': Y' \rightarrow \mathbb{P}_k^1$ defined by an affine equation $y^p-y=g(x)$ is isomorphic to $\phi$ if $g(x)=f(x)+h(x)^p-h(x)$ for some $h(x) \in K(x)$. At each ramification point, there is a filtration of the inertia group $\mathbb{Z}/p$, called the filtration of higher ramification groups (see, e.g., \cite{MR554237}, IV). Suppose $f(x)$ has $r$ poles: $\{P_1,\ldots,P_{r}\}$ on $\mathbb{P}^1_k$. Let $d_j$ be the order of the pole of $f(x)$ at $P_j$. One may assume that $f(x)$ has \emph{minimal form}; i.e., the number $d_j$ is prime to $p$ by Artin-Schreier theory. It is an easy exercise to show that $d_j$ is the \emph{ramification jump} at $P_j$. Let $e_j=d_j+1$. Then $e_j \ge 2$ and $e_j \not\equiv 1 \bmod p$. The ramification divisor of $\phi$ is $D:=\sum_{j=1}^{r} (p-1)e_j Q_j$ where $Q_j$ is the ramification point above $P_j$ (\cite{MR554237}, IV, Proposition 4). The $p$-rank of $Y$ is an integer $s$ such that the cardinality of $\text{Jac}[p](k)=p^s$.  Applying Riemann-Hurwitz formula (\cite{MR0463157}, IV, Corollary 2.4) and Deuring-Shafarevich formula (\cite{MR742696}, Corollary 1.8), we have the following lemma. 

\begin{lemma} [{\cite[Lemma 2.6]{MR2985514}}]
\label{lemmagenus}
The genus of $Y$ is $g_Y=((\sum_{j=1}^{r}e_j)-2)(p-1)/2$. The $p$-rank of $Y$ is $s_Y=(r-1)(p-1)$.
\end{lemma}

Throughout the paper, we define $d$ by $g=d(p-1)/2$. So, we have the identities: $d=(\sum_{j=1}^{r}e_j)-2$ and $\sum_{j=1}^{r}e_j=d+2=2g/(p-1)+2.$

\begin{remark}
The above lemma shows that all the Artin-Schreier $k$-curves with the same genus $g$ have the same $d$. That is the essential difference between $\mathbb{Z}/p$-curves and $\mathbb{Z}/q$-curves where $q\neq p$ is prime. For each $\mathbb{Z}/q$-curve, a branch point contributes $q-1$ to the degree of its ramification divisor. Thus, every $\mathbb{Z}/q$-cover of genus $g$ must have the same number of branch points.
\end{remark}

\begin{remark}
Lemma \ref{lemmagenus} also implies that for $p>2$, the moduli space $\mathcal{AS}_g$ is empty and thus is connected if $g$ is not a multiple of $(p-1)/2$. Hence, throughout the paper we will assume $g$ is a multiple of $(p-1)/2$ when $p>2$. 
\end{remark}

\subsection{Partitions of Integers}
\label{partition}

Fix a prime $p>0$ and an integer $d\ge 1$. Let $\Omega_d$ be the set of partitions of $d+2$ into positive integers $e_1,e_2, \ldots$ with each $e_j \not \equiv 1 \pmod p$. Without loss of generality, suppose $e_1 \ge e_2 \ge \ldots \ge e_n$. Define a partial ordering $\prec$ on $\Omega_d$: $\overrightarrow{E} \prec \overrightarrow{E'}$ if $\overrightarrow{E'}$ is a refinement of $\overrightarrow{E}$. Using this partial ordering, one can construct a directed graph $G_d$. The vertices of the graph correspond to the partitions $\overrightarrow{E}$ in $\Omega_d$. We say a partition $\overrightarrow{E}=\{e_1,\ldots,e_n\}$ of length $n$ is in the \emph{$n$-th level} of $G_d$, denoted $\overrightarrow{E} \in \Omega_{d,n-1}$ . There is an \emph{edge} from $\overrightarrow{E}$ to $\overrightarrow{E'}$ if and only if $\overrightarrow{E} \prec \overrightarrow{E'}$ and there is no partition lying strictly between them. It is easy to see that there are only two types of edges. The first type has the form $\{e\} \rightarrow \{e_1, e_2\}$ and the second type has the form $\{e\} \rightarrow \{e_1,e_2,e_3\}$ where $e_i \equiv \frac{p+1}{2} \pmod p$. We say there is a \emph{path} connecting $\overrightarrow{E}$ and $\overrightarrow{E'}$ if one can go from $\overrightarrow{E}$ to $\overrightarrow{E'}$ through some sequence of edges in $G_d$.

Below are some straightforward results about the set $\Omega_d$ and the graph $C_d$ from \cite{MR2985514} that will be used in \S\ref{connectednesschap}.

\begin{lemma}
\label{lemmaminimalpartition}
The set $\Omega_{d,0}$ is nonempty if and only if $p \nmid (d+1)$. If $p \nmid (d+1)$, then $\Omega_{d,0}$ contains one partition $\{d+2\}$ which is an initial vertex of $G_d$. If $p \mid (d+1)$, then $\Omega_{d,1}$ consists of $\lceil(d+1)/(p-2)/2p \rceil$ partitions, and every vertex of $G_d$ is larger than one of these.
\end{lemma}

\begin{lemma}
\label{lemmamaximalpartition}

Let $p\ge 3$. A partition is maximal if and only if its entries are all equal to two or three. Every integer $r+1$ with $(d-1)/3 \le r \le d/2$ occurs exactly once as the length of a maximal partition. There are $\lfloor d/2 \rfloor-\lceil(d-4)/3\rceil$ maximal partitions. There is a unique maximal partition if and only if $d \in \{1,2,3,4\}$. 
\end{lemma}

\begin{proof}
Lemma $2.4$ of \cite{MR2985514}.
\end{proof}

The next proposition, which shows that $G_d$ is connected when $k$ has small characteristic, will be used to show $\mathcal{AS}_g$ is connected.

\begin{proposition}
\label{connectednessfirstgraph}
If $k$ has characteristic $2$ or $3$, then $G_d$ is connected for all $d$.

\end{proposition}

\begin{proof}

\textit{Case $p=2$:} Since $d+2=2g/(2-1)+2$ is even for any $g$, the partition $\overrightarrow{E}=\{d+2\}$ is the unique minimal partition of $\Omega_d$ by Lemma \ref{lemmaminimalpartition}. Thus, $G_d$ is connected.

\textit{Case $p=3$:} 

\begin{itemize}

\item \emph{If $d+2 \not\equiv 1 \bmod 3$}, all partitions $\overrightarrow{E} \in \Omega_d$ satisfy $\{d+2\} \prec \overrightarrow{E}$.  Thus $G_d$ is connected.

    \item \emph{If $d+2 \equiv 1 \bmod 3$}, i.e. $d+2=3n+1$ for some $n \ge 1$, all the minimum partitions are of the form $\overrightarrow{E}=\{3N+2,3M+2\}$ where $N+M=n-1$ by Lemma \ref{lemmaminimalpartition}. All the branches of the graph $G_d$ starting from these $\overrightarrow{E}_i$ must have the same maximal partition $\{3,\ldots,3,2,2\}$ by Lemma \ref{lemmamaximalpartition}. Thus, $G_d$ is connected.

\end{itemize}

\end{proof}

\begin{remark}
The above proposition is not true when $k$ has characteristic $p>3$. For instance, when $p=5$ then $\Omega_4$ has only two partitions: $\{3,3\}$ and $\{2,4\}$. Since it is clear that $\{3,3\} \not\prec \{2,4\}$ or vice versa, $G_4$ is disconnected.
\end{remark}

\begin{example}
 
 Let $p=5$ and $d=7$. Here is the graph $G_{7}$ for $\Omega_{7}$.

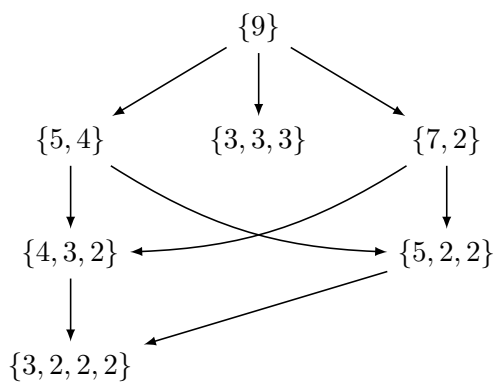
\begin {center}
\begin {tikzpicture}[-latex ,auto ,node distance =1.5 cm and 2.5cm ,on grid ,
semithick ] 
\node (C)
{$\{9\}$};
\node (A) [below left=of C] {$\{5,4\}$};
\node (B) [below right =of C] {$\{7,2\}$};
\node (D) [below =of C] {$\{3,3,3\}$} ;
\node (E) [below =of A] {$\{4,3,2\}$} ;
\node (F) [below =of B] {$\{5,2,2\}$} ;
\node (G) [below =of E] {$\{3,2,2,2\}$} ;
\path (C) edge node[below =0.15 cm] {} (A);
\path (C) edge node[below =0.15 cm] {} (D);
\path (C) edge node {} (B);
\path (B) edge node {} (F);
\path (B) edge [bend left =15] node[below =0.15 cm] {} (E);
\path (A) edge node {} (E);
\path (A) edge [bend right =15] node[below =0.15 cm] {} (F);
\path (E) edge node {} (G);
\path (F) edge node {} (G);

\end{tikzpicture}

\end{center}
\end{example}

\subsection{Moduli of Artin-Schreier Curves and Partitions of Integers}
\label{modulipartition}



Recall that $\mathcal{AS}_g$ is the moduli space that parameterizes Artin-Schreier curves of genus $g$. For detail, see \cite[Section 3.1]{MR2985514}. From the fact that Artin-Schreier curves of the same genus $g$ have the same invariant $d+2$, Pries and Zhu show that $\mathcal{AS}_g$ can be partitioned by locally closed strata corresponding to elements of $\Omega_d$. In particular, the partition $\overrightarrow{E}=\{e_1,\ldots,e_r\}$  of $d+2$ is associated with the stratum $\mathcal{AS}_{g,\overrightarrow{E}}$ (or $\Gamma_{\overrightarrow{E}}$) which is the collection of all the points of $\mathcal{AS}_g$ that represent Artin-Schreier curves with ramification divisors of the form $\sum_{j=1}^{r} (p-1)e_j P_j$, where $P_j$ are distinct points of $\mathbb{P}^1_k$.

\begin{proposition}
\label{propdimension}

The moduli space $\mathcal{AS}_g$ is partitioned by a set of irreducible strata that is in one-to-one correspondence with elements of $\Omega_d$. If $\overrightarrow{E}=\{e_1, e_2, \ldots, e_{r}\}$, then the dimension $d_{\overrightarrow{E}}$ over $k$ of the irreducible stratum $AS_{g,\overrightarrow{E}}$ is $d-1-\sum_{j=1}^{r} \lfloor (e_j-1)/p \rfloor$.

\end{proposition}

\begin{proof}
Corollary 3.11 of \cite{MR2985514}.
\end{proof}

As a result of the above proposition, one has a necessary condition for one stratum to be in the closure of another stratum by comparing their dimensions. 

\begin{lemma}
For $a \in \mathbb{Z}_{>0}$, let $\bar{a}$ be the integer so that $\bar{a} \equiv a \pmod p$ and $0 \le \bar{a} < p$.

\begin{enumerate}
  \item If the edge from $\overrightarrow{E}_1$ to $\overrightarrow{E}_2$ is of the form $\{e\}$ $\rightarrow$ $\{e_1,e_2\}$ with $2 < \bar{e}_2+\bar{e}_2 \le p$ and $\bar{e}_1\bar{e}_2 \neq 0$, then $\dim_k (\Gamma_{\overrightarrow{E}_1})=\dim_k (\Gamma_{\overrightarrow{E}_2})$ and $\Gamma_{\overrightarrow{E}_1}$ is not in the closure of $\Gamma_{\overrightarrow{E}_2}$ in $\mathcal{AS}_g$.
  \item In all other cases, $\dim_k (\Gamma_{\overrightarrow{E}_1})=\dim_k (\Gamma_{\overrightarrow{E}_2})-1$.
\end{enumerate}
\end{lemma}

\begin{proof}

Lemma 4.5 of \cite{MR2985514}.

\end{proof}

\begin{remark}
Suppose we have two strata connected by an edge in $G_d$. We will see later in Corollary \ref{openquestion1} that comparing dimensions also gives a sufficient condition for one stratum lie in the closure of another.
\end{remark}

\begin{remark}

The original statement of Lemma 4.5 of \cite{MR2985514} does not have the condition that $\overline{e}_1 \overline{e}_2 \neq 0$, which is a mistake. Since, for instance, if $p=5$, $\overrightarrow{E}_1=\{7\}$ and $\overrightarrow{E}_2=\{5,2\}$, then $\dim_k (\Gamma_{\overrightarrow{E}_1})=3=\dim_k (\Gamma_{\overrightarrow{E}_2})-1$.

\end{remark}

Finally, it is known when the moduli space $\mathcal{AS}_g$ is irreducible.

\begin{corollary}
\label{corollaryirreducible}

The moduli space $\mathcal{AS}_g$ is irreducible in exactly the following cases: \inlinerom{\item $p=2$; or \item $g=0$ or $g=(p-1)/2$; or \item $p=3$ and $g=2,3,5.$}
\end{corollary}

\begin{proof}
Corollary 1.2 of \cite{MR2985514}
\end{proof}

\begin{remark}

The above corollary implies that the moduli space $\mathcal{AS}_g$ is connected when $p=2$. 

\end{remark}

\section{Strata lying in closures of the other strata}
\label{relation}

In this section, we introduce some techniques to study whether a stratum lies in the closure of certain other ones in $\mathcal{AS}_g$.  This involves deformations of wildly ramified curves with non-constant branch locus.

\subsection{Some Deformation Results}
\label{closure}

Let us first prove a general result for $k$-varieties.

\begin{lemma}
\label{propmorphism}

Suppose $A$ is an affine ring of dimension $d$ over the field $k$ and $m$ is a maximal ideal of $A$. Suppose moreover that we have a chain of prime ideals $(0) \subsetneq I \subsetneq m$ of $A$ with $d>\dim(A/I)=d_1>0$. Then there exists a $k$-algebra homomorphism

\[ \phi: A \rightarrow k[[t]],\]

\noindent where $\phi^{-1}((t))=m$ and $\phi^{-1}(0) \nsupseteq 
 I$.
\end{lemma}

\begin{proof}

By Noether's Normalization Lemma (see \cite{MR1322960}, \S$13$, Theorem $13.3$), $A$ is isomorphic as a $k$-algebra to $k[x_1, \ldots ,x_d, y_1, \ldots, y_l]/(p_1,\ldots, p_n)$ where each $p_i$ is irreducible and of the form

\[p_i(X,y_i)= y_i^{L_i}+Q^i_{k_i-1}(X)y_i^{L_i-1}+ \dots +Q^i_0(X), \]

\noindent with $Q^i_k(X)=\sum_{\nu \in \mathbb{N}^d}\alpha_{\nu}X^{\nu} \in k[x_1,\ldots,x_d]$. In other words, $A$ is integral over $k[x_1,\ldots,x_d]$. With a good choice of $x_i$ and $y_j$, one may assume that $m=(x_1, \ldots ,x_d, y_1, \ldots, y_l)$, $I\cap k[x_1,\ldots,x_d]=(x_{d_1+1},\ldots, x_d) \neq (0)$ and $Q^j_0(X)$ has no constant term for all $j$. Define a map

\[ \phi: A=k[x_1, \ldots ,x_d, y_1, \ldots, y_m]/(p_1,\ldots, p_m) \rightarrow \overline{k((T))},\]

\noindent by mapping each $x_i$ to $T$ and each $y_j$ to a root of 

\[p_j(\phi(X),Y_j)=Y_j^{L_i}+Q^i_{k_i-1}(\phi(X))Y_j^{L_i-1}+ \dots +Q^i_0(\phi(X))\]

\noindent in $\overline{k((T))}$. Since $Q^i_0(X)$ has no constant term, $Q^i_0(\phi(X))$ is a multiple of $T$, i.e. $v_{T} (Q^i_0(\phi(X))\in \mathbb{Z}_+$. Let $L/k((T))$ be the splitting field of the polynomials $p_j(\phi(X),Y_j)$. Suppose $[L:k((T))]=s$. Then there exists $t \in L$ such that $L=k((t))$. Because $k((T))$ is complete with respect to the $T$-adic valuation, the valuation $v_T$ extends uniquely to $k((t))$ (see \cite{MR1697859}, II, \S$4$, Theorem $4.8$). Moreover, the extension $v_t$ is given by the formula:

\[v_t(\alpha)=\dfrac{1}{s}v_T(N_{k((t))/k((T))}(\alpha)).\]

\noindent Hence, $v_t(\phi(y_j))=\frac{1}{s}v_T(Q^j_0(X)^{(s/L_j)})=\frac{1}{L_j}v_T(Q^j_0(X)) \in \frac{1}{s}\mathbb{Z}_+$ for each $j=1,2,\ldots m$. Thus, the map $\phi$ restricts to a $k$-algebra homomorphism:

\[ \phi: k[x_1, \ldots ,x_d, y_1, \ldots, y_m]/(p_1,\ldots, p_m) \rightarrow k[[t]],\]

\noindent where each of the generators of $A$ is mapped to a multiple of a positive integer power of $t$. The proposition follows easily. 

\end{proof}

\begin{proposition}
\label{propclosurevariety}

Suppose $U$ and $V$ are two irreducible disjoint subsets of a $k$-variety $X$ and $0 < \dim U= \dim \overline{U} < \dim V=:d$. If $U$ is contained in the closure of $V$ then for each point $u$ of $U$, there exists a $k[[t]]$-point of $X$ whose generic point is in $V$ and whose closed point is $u$.

\end{proposition}

\begin{proof}

Suppose $\spec A \subsetneq (\overline{U}\cup V)$ is an irreducible affine variety of dimension equal to $d$ that containing $u$. Then there exists a chain of prime ideals $(0) \subsetneq I \subsetneq m$ of $A$ where $\overline{U}\cap \spec A=\mathcal{V}(I)$ and $u=\mathcal{V}(m)$ and $\dim A>\dim(A/I)=:d_1>0$ and $\dim(A/m)=0$. By Lemma \ref{propmorphism}, there exists a $k$-algebra homomorphism: 

\[ \tilde{\theta}: A \rightarrow k[[t]],\]

\noindent where $\tilde{\theta}^{-1}(t)=m$ and $\tilde{\theta}^{-1}(0) \nsupseteq 
 I$. Thus $\tilde{\theta}$ induces a morphism of $k$-varieties:
 
 \[ \theta: \spec k[[t]] \rightarrow \spec A,\]
 
\noindent where $\theta(\mathcal{V}(t))=\mathcal{V}(m)$ and $\theta(\mathcal{V}(0)) \not\in  \mathcal{V}(I)$, that is $\theta(\mathcal{V}(0))\in (V\cap \spec A)$. Hence, $\theta$ is a $k[[t]]$-point whose generic point is in $V$ and whose closed point is $u$.

\end{proof}

The next fact about germs of Artin-Schreier curves will help us to reduce the problem of constructing deformations for an Artin-Schreier curve to the local case.

\begin{proposition}
\label{localAS}
Suppose $\phi: Y \rightarrow \mathbb{P}^1_k$ is an Artin-Schreier cover and $P \in \mathbb{P}^1_k$ is a branch point of $\phi$. Then the localization of $\phi$ at $P$ is determined by the ramification jump at $P$. 
\end{proposition}

\begin{proof}

See Lemma $2.1.2$ of \cite{MR2016596}.

\end{proof}

The next proposition, which is about the relation between the irreducible strata of $\mathcal{AS}_g$ and deformations of wildly ramified curves, will be used intensively in \S\ref{technical}.

\begin{proposition}
\label{propreduce}
Suppose $\overrightarrow{E}_1=\{e_1,e_2,\ldots,e_n\} \prec \overrightarrow{E}_2=\{f_{11},f_{12},\ldots,f_{1m_1},f_{21},f_{22},\ldots,f_{nm_n}\}$, with $\{e_i\}\prec \{f_{i1},\ldots, f_{im_i}\}$ for $i=1,2,\ldots,n$. Then $\Gamma_{\overrightarrow{E}_1}$ is in the closure of $\Gamma_{\overrightarrow{E}_2}$ if and only if there exist $n$ Artin-Schreier covers $\Phi_1, \Phi_2, \ldots, \Phi_n$ over $S=\spec k[[t]]$, where the generic fiber of $\Phi_i$ yields a $k((t))$-point of $\Gamma_{\{f_{1i},\ldots, f_{im_i}\}}$ and the special fiber yields a $k$-point of $\Gamma_{\{e_i\}}$.

\end{proposition}

\begin{proof}

The ``if'' conditions of the above proposition can be rephrased as follows. There are $n$ Artin-Schreier covers $\Phi_i:Y_i \rightarrow Z_i \simeq \mathbb{P}^1_S$, for $i=1,2, \ldots,n$, whose special fiber is an Artin-Schreier curve over $k$, branched at a point $b_i$ with ramification jump $e_i-1$, and whose generic fiber is an Artin-Schreier curve over $k((t))$, branched at $m_i$ points that specialize to $b_i$ and which have ramification jumps $f_{i1}-1,\ldots, f_{im_i}-1$. Consider a point in $\Gamma_{\overrightarrow{E}_1}$ that represents an Artin-Schreier cover $\phi: Y \rightarrow \mathbb{P}^1_k$, which is wildly ramified at $y_1,\ldots,y_n \in Y$ over $b_1,\ldots,b_n \in \mathbb{P}^1_k$, respectively,  with ramification jumps $e_1-1,\ldots,e_n-1$, respectively. Let $\hat{\phi}_{j}$ be the germ of $\phi$ at $y_j$. The localization of the special fiber of $\Phi_j$ at its branch point has ramification jump $e_{j-1}$ by the hypothesis. Since the germ of an Artin-Schreier cover at its branch point is uniquely determined by the ramification jump at this point by Proposition \ref{localAS}, the germ in the previous sentence is isomorphic to $\hat{\phi}_j$. Thus, $\Phi_j$ is a deformation of a curve such that the localization at its branch point is isomorphic to $\hat{\phi}_j$.  Using formal patching, see e.g., \cite[Theorem 3.3.4]{MR1767273}, one can construct a deformation $\Phi: \mathscr{Y}\rightarrow \mathbb{P}_S^1$ of $\phi$ locally via the deformations $\Phi_j$ of $\hat{\phi}_j$. It follows from the construction that $\Phi$ is a $\mathbb{Z}/p$-cover whose special fiber is isomorphic to $\phi$ and whose generic fiber can be represented by a $k((t))$-point of $\Gamma_{\overrightarrow{E}_2}$. Whence, the point in $\Gamma_{\overrightarrow{E}_1}$ that indexes $\phi$ is in the closure of $\Gamma_{\overrightarrow{E}_2}$.

Conversely, if $\Gamma_{\overrightarrow{E}_1}$ is in the closure of $\Gamma_{\overrightarrow{E}_2}$, then $\dim \Gamma_{\overrightarrow{E}_1}< \dim \Gamma_{\overrightarrow{E}_2}$. Thus, by Proposition \ref{propclosurevariety}, for each $k$-point of $\Gamma_{\overrightarrow{E}_1}$ that represents a cover $\phi: Y \rightarrow \mathbb{P}^1_k$, there exists an Artin-Schreier curve $\Phi: \mathscr{Y} \rightarrow \mathbb{P}^1_S$ so that the generic fiber yields a $k((t))$-point of $\Gamma_{\overrightarrow{E}_2}$ and the special fiber is isomorphic to $\phi$. Suppose $\phi$ is wildly ramified at $y_1,\ldots,y_n \in Y$ over $b_1,\ldots,b_n \in \mathbb{P}^1_k$, respectively,  with ramification jumps $e_1-1,\ldots,e_n-1$, respectively. We can get each of the local deformations $\Phi_i$ simply by restricting $\Phi$ to the spectrum of the complete local ring $\hat{\mathcal{O}}_{\mathscr{Y}, y_i}$.
 
\end{proof}

\begin{remark}
The above proposition implies that it suffices to study the deformations of germs or, in other words, to reduce to the case $\overrightarrow{E}_1$ has only one entry.
\end{remark}




The following corollary will be used in \S\ref{disconnectedsection} to prove a result about disconnectedness of $\mathcal{AS}_g$.

\begin{corollary}
\label{corclosureintersect}

Let $\overrightarrow{E}_1$ and $\overrightarrow{E}_2$ be as in the preceding proposition. Suppose $\Gamma_{\overrightarrow{E}_1}$ is not in the closure of $\Gamma_{\overrightarrow{E}_2}$. Then $\Gamma_{\overrightarrow{E}_1}$ is disjoint from the closure of $\Gamma_{\overrightarrow{E}_2}$.

\end{corollary}

\begin{proof}

Suppose that there is a $k$-point of $\Gamma_{\overrightarrow{E}_1}$, which corresponds to a cover $\phi: Y \rightarrow \mathbb{P}^1_k$, that is in the closure of $\Gamma_{\overrightarrow{E}_2}$. Then, there is a deformation $\Phi$ over $S=\spec k[[t]]$ whose generic fiber yields a $k((t))$-point of $\Gamma_{\overrightarrow{E}_2}$ and its special fiber is isomorphic to $\phi$. Thus, for each $e_i$,  there exists a deformation over $S$ from a $k((t))$-point of $\Gamma_{\{f_{1i},\ldots, f_{im_i}\}}$ to a $k$-point of $\Gamma_{\{e_i\}}$ formed by restricting $\Phi$. Hence, by the preceding proposition, $\Gamma_{\overrightarrow{E}_1}$ is in the closure of $\Gamma_{\overrightarrow{E}_2}$.

\end{proof}

Motivated by Proposition 4.1 of  \cite{MR2985514} and with the aim to find sufficient conditions for the moduli space $\mathcal{AS}_g$ to be connected, we construct deformations that will prove Theorem \ref{connectedness}.

\begin{theorem}
\label{deformationtheorem}

 Let $e \not \equiv 1 \pmod{p}$ be an integer and $\{e_1, e_2, \ldots, e_n\}$ be a partition of $e$ such that $e_i \not\equiv 1 \pmod{p}$ for all $i$. Fix a projective line $\mathbb{P}^1_{k[[t]]}$ and a $k$-point $b$ of $\mathbb{P}^1_{k[[t]]}$. Identify a projective line $\mathbb{P}^1_{k((t))}$ with the generic fiber of $\mathbb{P}^1_{k[[t]]}$. Then there exists a $\mathbb{Z}/p$-Galois cover $\psi_S$ over $S=\spec k[[t]]$ whose special fiber is an Artin-Schreier cover over $k$, branched only at $b$ with ramification jump $e-1$, and whose generic fiber is isomorphic to an Artin-Schreier cover $\psi_{\eta}$ over $\mathbb{P}^1_{k((t))}$ that is branched at $n$ distinct points $a_1, a_2, \ldots, a_n$ which specialize to $b$ and which have ramification jumps $e_i-1$ at $a_i$ in the following situations:

\begin{enumerate}[label=(\alph*)]
\item $n=2$ and $\lfloor \frac{e-1}{p} \rfloor > \lfloor \frac{e_1-1}{p} \rfloor + \lfloor \frac{e_2-1}{p} \rfloor$, or
\item $n=3$ and $e_i \equiv \frac{p+1}{2} \pmod p$ for all $i$, or
\item $n=3$ and $e_i=p-1$ for some $i$, or
\item $n=p-m+1$ and $e_i=m+1$ for all $i$, or
\item $n=4$ and $\{e_1,e_2,e_3,e_4\}=\{3,2,2,2\}$ or $\{3,3,2,2\}$ where $p=5$.

\end{enumerate}


\end{theorem}

\begin{proof}

The proof will be provided in \S \ref{technical}. 

\end{proof}

\begin{remark}
Proposition 4.1 of \cite{MR2985514} is a special case of Theorem \ref{deformationtheorem} part $(a)$.
\end{remark}

\subsubsection{The Oort-Sekiguchi-Suwa Component}

Suppose $\Psi$ is a $\mathbb{Z}/p$-cover over $R=\mathcal{W}(k)[\zeta_p]$. In \cite{MR1767273}, Bertin and M\'{e}zard study the \emph{versal deformation ring} $R_{\Psi}$ whose spectrum is the formal deformation space of $\Psi$. They construct explicitly a family of infinitesimal deformations for $\Psi$ that is parameterized by an irreducible component of the formal deformation space, called \emph{the Oort-Sekiguchi-Suwa component}. It turns out that the dimension of the component is equal to the dimension of the versal deformation ring. In particular, if $\Psi$ has ramification jump $m=pq-l$, where $l \in [1,p[$, we have:

\[\dim_{\text{Krull}} R_{\Psi}=\begin{cases} 
      q & l \neq 1 \\
      q+1 & l=1 
   \end{cases}
\]

\noindent For more details, see Theorem 5.3.3 of \cite{MR1767273}. Most of the deformations that we will construct to prove Theorem \ref{deformationtheorem} are explicit examples of ones that do \emph{not} lie in the Oort-Sekiguchi-Suwa component. In particular, only the local deformation $\psi_S$ in part $(a)$ of Theorem \ref{deformationtheorem}, under the additional condition $e_1=p$ and $e_2 \le p$ is an element of the component.

In the rest of this subsection, we shall realize the characteristic $p$ fibers of the deformations in the Oort-Sekiguchi-Suwa component as special cases of ones that are used by Pries and Zhu in the proof of Proposition 4.1 of \cite{MR2985514}. Consider an Artin-Schreier cover $\phi$ over $k$ defined by an affine equation:

\begin{equation}
\label{bertinmezardArtinschreier}
    y^p-y=t^{-m},
\end{equation}

\noindent where $m=pq-l$, $l \in [1,p[$. One can think of $\phi$ as a $k$-point of $\Gamma_{\{m+1\}}$. We will study the case $l=1$ and $l \ne 1$ separately.

\textit{Case $l=1$}. Let $\zeta=yt^q$, then (\ref{bertinmezardArtinschreier}) can be written as:

\begin{equation}
\label{firstzetabertinmezardArtinschreier}
    \zeta^p-t^{(p-1)q}\zeta=t.
\end{equation}

\noindent Consider the Artin-Schreier cover $\Phi$ over $S=\spec k[[x_1, \ldots, x_q]]$ given by the following affine equation:

\begin{equation}
\label{firstbertinmezarddeformation}
    \zeta^p-(t^q+x_1t^{q-1}+\dots+x_q)^{p-1}\zeta=t.
\end{equation}

\noindent The special fiber of $\Phi$ is clearly birationally equivalent to $\phi$. On the generic fiber of $\Phi$, we define $Z=\zeta/(t^q+x_1t^{q-1}+\dots+x_q)$. Equation (\ref{firstbertinmezarddeformation}) can be written in familiar form:

\begin{equation}
\label{firstdeformcannonical}
    Z^p-Z=\dfrac{t}{(t^q+x_1t^{q-1}+\dots+x_q)^p}=:F(t).
\end{equation}

\noindent Let us study $a(t)=t^q+x_1t^{q-1}+\dots+x_q$. One can show that the polynomial is separable over $k[[t]][[x_1,\ldots,x_q]]$. Suppose $u_0,u_1, \ldots, u_q$ are roots of $x^q+x_1t^{q-1}+\dots+x_q$ over a fixed algebraic closure of $k[[t]][[x_1,\ldots,x_q]]$. Then the right hand side of the above equation is:

\[F(t)=\dfrac{t}{\prod_{i=1}^q(t-u_i)^p}.\]

\noindent Thus, 

\[((t-u_j)^pF(t))'=\dfrac{1}{\prod_{\substack{i=1 \\ i \ne j}}^q(t-u_i)^p}\]

\noindent is clearly not equal to $0$ or undefined when we plug in $t=u_j$ Hence, the minimal form of Artin-Schreier equation (\ref{firstdeformcannonical}) has precisely $q$ distinct poles of order $p-1$. It follows from Lemma IV.2.3 of \cite{MR1011987} that $\Phi$ is a deformation from a point of $\Gamma_{\{pq\}}$ to a point of $\Gamma_{\overrightarrow{E}_2}$ where $\overrightarrow{E}_2=\underbrace{\{p,p,\ldots,p\}}_{q}$.

\textit{Case $l \ne 1$}. Again, let $\zeta=yt^q$. This time, equation (\ref{bertinmezardArtinschreier}) is written as:

\begin{equation}
\label{secondzetabertinmezardArtinschreier}
    \zeta^p-t^{(p-1)q}\zeta=t^l.
\end{equation}

\noindent Consider the Artin-Schreier cover $\Phi$ over $S=k[[x_1, \ldots, x_{q-1}]]$ defined by the equation:

\begin{equation}
\label{secondbertinmezarddeformation}
    \zeta^p-(t^q+x_1t^{q-1}+\dots+x_{q-1}t)^{p-1}\zeta=t^l.
\end{equation}

\noindent Again, the special fiber is birationally equivalent to $\phi$. On the generic fiber, just like when $l=1$, define $Z=\zeta/(t^q+x_1t^{q-1}+\dots+x_{q-1}t)$. Equation (\ref{secondbertinmezarddeformation}) becomes:

\begin{equation}
\label{seconddeformationcannonical}
    Z^p-Z=\dfrac{t^l}{(t^q+x_1t^{q-1}+\dots+x_{q-1}t)^p}.
\end{equation}

\noindent By a similar argument as in the case $l=1$, the minimal form of Artin-Schrier equation (\ref{seconddeformationcannonical}) has precisely a pole at $0$ of order $p-l$ and $q-1$ other distinct poles of order $p-1$. Thus, again by Lemma IV.2.3 of \cite{MR1011987}, $\Phi$ is a deformation from a point of $\Gamma_{\{m+1\}}$ to a point of $\Gamma_{\overrightarrow{E}_2}$ where $\overrightarrow{E}_2=\{\underbrace{p,p,\ldots,p}_{q-1},p-l+1\}$.

Let us compare those deformations above with one in \cite{MR2985514}. Pries and Zhu's deformations in Proposition $4.1$ of \cite{MR2985514} are defined by equations of the form:

\begin{equation}
\label{RPdeform}
y^p-y=\dfrac{1}{x^{e_2-1}(x-t)^{e_1}},
\end{equation}

\noindent where $p \mid e_1$, say $e_1=pn_1$. One can check the minimal form of the Artin-Schreier equation (\ref{RPdeform}) has a pole at $0$ of order $e_2-1$ and at $t$ of order $e_1-1$. Thus, this is a deformation from a point of $\Gamma_{\{e\}}$ to a point of $\Gamma_{\{pn_1,e_2\}}$. One can get all elements of the Oort-Sekiguchi-Suwa component by iterating Pries and Zhu's deformation.

\begin{remark}
When $p=2$, if $\overrightarrow{E}=\{e_1,e_2,\ldots,e_n\} \in \Omega_d$ then $e_i$ must all be even, or else $e_i$ would be congruent to $1$ modulo $p$. Thus, all deformations are formed by repeated applications of equation (\ref{RPdeform}).
\end{remark}

\subsection{Some Closure Results}
\label{closuresection}

In this section, we show that the combinatorial data in the graph $G_d$ gives partial information about how irreducible strata of $\mathcal{AS}_g$ relate. In fact, we will see in \S\ref{connectednesschap} that the graph $G_d$ gives complete information about this question when $p=2$ or $p=3$.

Let us fix a prime $p$ and construct a directed graph $C_d$. The vertices of the graph correspond to the partitions $\overrightarrow{E}$ in $\Omega_d$. There is an edge from $\overrightarrow{E}$ to $\overrightarrow{E}'$ if and only if $\overrightarrow{E}\prec \overrightarrow{E}'$, and $\Gamma_{\overrightarrow{E}}$ lies in the closure $\Gamma_{\overrightarrow{E}'}$. We will show in \S\ref{connectednesschap} that for $p=2$ or $p=3$, the graph $G_d$ is a subgraph of $C_d$.
 
In general topology, if one irreducible subset of a space lies in the closure of another, then they are contained in the same connected component of that space. Thus, if $C_d$ is connected then so is $\mathcal{AS}_g$. From now on, we will study the connectedness of $C_d$.


\begin{theorem}
\label{theoremclosure}

Suppose $\overrightarrow{E}_1 \prec \overrightarrow{E}_2$. Then there is an edge from the former to the latter in $C_d$ if each entry of $\overrightarrow{E}_1$ splits into entries of $\overrightarrow{E}_2$ of the form:

    \begin{enumerate}[label=(\alph*)]
    \item $\{e\} \rightarrow \{e_1,e_2\}$ where  $\lfloor \frac{e-1}{p} \rfloor > \lfloor \frac{e_1-1}{p} \rfloor + \lfloor \frac{e_2-1}{p} \rfloor$.
    \item $\{e\} \rightarrow \{e_1,e_2,e_3\}$ where $e_i \equiv \frac{p+1}{2} \pmod p$ for $i=1,2,3$.
    \item $\{e\} \rightarrow \{p-1, e_1,e_2\}$. 
    \item $\{(n+1)(p-n+1)\} \rightarrow \underbrace{\{n+1, n+1, \ldots,n+1\}}_{p-n+1}$ where $1\le n \le p-1$.
    \item $\{9\} \rightarrow \{3,2,2,2\}$ or $\{10\} \rightarrow \{3,3,2,2\}$ where $p=5$.
    \end{enumerate}

\end{theorem}

\begin{proof}

It suffices to show that if $\overrightarrow{E}_1$ and $\overrightarrow{E}_2$ satisfy the conditions listed above then $\Gamma_{\overrightarrow{E}_1}$ is in the closure of $\Gamma_{\overrightarrow{E}_2}$. By Proposition \ref{propreduce}, we can reduce to the case $\overrightarrow{E}_1=\{e\}$.

$(a)$ Let $\phi_o: Y_o \rightarrow \mathbb{P}_k ^1$ be an Artin-Schreier cover that corresponds to a $k$-point of $\Gamma_{\overrightarrow{E}_1}$. The element $e$ in the partition $\Gamma_{\vec{E}_1}$ indicates the ramification jump of a branch point $b \in \mathbb{P}_k^1$ so that the ramification jump of $\phi_o$ above $b$ is $e-1$.

Let $S=\spec(k[[t]])$. By Theorem \ref{deformationtheorem} $(a)$, there exists an Artin-Schreier cover $\phi_S$ over $S$ whose special fiber is isomorphic to $\phi_o$ and whose generic fiber is branched at two points that specialize to $b$ and that have ramification jumps $e_1-1$ and $e_2-1$. Furthermore, the ramification divisor is otherwise constant. Thus the generic fiber $\phi_S$ has partition $\vec{E}_2$. Thus $\Gamma_{\vec{E}_1}$ is in the closure of $\Gamma_{\vec{E}_2}$ in $\mathcal{AS}_g$.

Parts $(b),(c),(d),(e)$, are proved using similar arguments as in part $(a)$ and Theorem \ref{deformationtheorem} parts $(b),(c),(d),(e)$.
\end{proof}

\begin{remark}
Item $(e)$ of the above lemma is used to lower the bound for the genus $g$ so that the moduli space $\mathcal{AS}_g$ is connected when the base field has characteristic $5$.
\end{remark}

Combining the result above with the discussion in Section \ref{partition} about the two types of edges in the directed graph $G_d$, we have the complete answer for Open Question 1 of \cite{MR2985514}: the necessary and sufficient conditions on the edge $\{e\} \rightarrow \{e_1,e_2\}$ or the edge $\{e\} \rightarrow \{e_1,e_2,e_3\}$ in $G_d$ for $\Gamma_{\vec{E}_1}$ to be in the closure of $\Gamma_{\vec{E}_2}$ in $\mathcal{AS}_g$.

\begin{corollary}
\label{openquestion1}

\begin{enumerate}[label=(\alph*)]
\item If $\vec{E}_1 \prec \vec{E}_2$ with an edge of the form $\{e\} \rightarrow \{ e_1, e_2\}$ in $G_d$ from $\vec{E}_1$ to $\vec{E}_2$, then $\Gamma_{\vec{E}_1}$ is in the closure of $\Gamma_{\vec{E}_2}$ if and only if $\lfloor \frac{e-1}{p} \rfloor > \lfloor \frac{e_1-1}{p} \rfloor + \lfloor \frac{e_2-1}{p} \rfloor$.

\item If $\vec{E}_1 \prec \vec{E}_2$ with an edge of the form $\{e\} \rightarrow \{ e_1, e_2, e_3\}$ in $G_d$ from $\vec{E}_1$ to $\vec{E}_2$, then $\Gamma_{\vec{E}_1}$ is in the closure of $\Gamma_{\vec{E}_2}$. 

\item Suppose $\vec{E}_1 \prec \vec{E}_2$ with an edge between them in $G_d$. Then $\Gamma_{\vec{E}_1}$ is in the closure of $\Gamma_{\vec{E}_2}$ if and only if $\dim_k(\Gamma_{\vec{E}_1})< \dim_k(\Gamma_{\vec{E}_2})$.

\end{enumerate}

\end{corollary}

\begin{proof}

\textit{For edge of the form $\{e\} \rightarrow \{ e_1, e_2\}$:}. The ``if'' direction follows from part $(a)$ of Theorem \ref{theoremclosure}. Conversely, if $\lfloor \frac{e-1}{p} \rfloor = \lfloor \frac{e_1-1}{p} \rfloor + \lfloor \frac{e_2-1}{p} \rfloor$, then by theorem 3.10 of \cite{MR2985514}, dim$_k$ $(\Gamma_{\overrightarrow{E}_1})$=dim$_k$ $(\Gamma_{\overrightarrow{E}_2})$ and $\Gamma_{\overrightarrow{E}_1}$ is not in the closure of $\Gamma_{\overrightarrow{E}_2}$ in $\mathcal{AS}_g$.

\textit{For edge of the form $\{e\} \rightarrow \{ e_1, e_2, e_3\}$:} In order to have an edge of type $\{e\} \rightarrow \{e_1,e_2,e_3\}$ in $G_d$, we must have $e_1+e_2 \equiv 1 \pmod p$, $e_1+e_3 \equiv 1 \pmod p$ and $e_2+e_3 \equiv 1 \pmod p$. Thus, $e_1 \equiv e_2 \equiv e_3 \equiv \frac{p+1}{2} \pmod p$. This is then an immediate result of part $(b)$ of Theorem \ref{theoremclosure}.

Since an edge in $G_d$ is either of a type in part (a) or type in part (b), part (c) easily follows.

\end{proof}




\section{Connectedness}
\label{connectednesschap}

In this section, we will prove Theorem \ref{deformationtheorem}. Let's take a look at one example.

\begin{example}

Let $p=5$ and $g=14$. Then $d=7$. Using Lemma \ref{theoremclosure}, we can draw a subgraph of $C_7$.

 
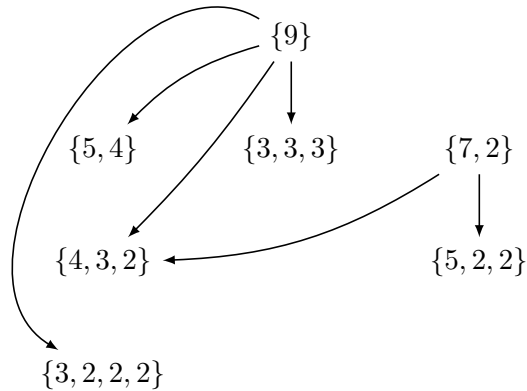
\begin {center}
\begin {tikzpicture}[-latex ,auto ,node distance =1.5 cm and 2.5cm ,on grid ,
semithick] 
\node (C)
{$\{9\}$};
\node (A) [below left=of C] {$\{5,4\}$};
\node (B) [below right =of C] {$\{7,2\}$};
\node (D) [below =of C] {$\{3,3,3\}$} ;
\node (E) [below =of A] {$\{4,3,2\}$} ;
\node (F) [below =of B] {$\{5,2,2\}$} ;
\node (G) [below =of E] {$\{3,2,2,2\}$} ;
\path (C) edge [bend right =15] node[below =0.15 cm] { } (A);
\path (C) edge node[below =0.15 cm] { } (D);
\path (B) edge [bend left =15] node[below =0.15 cm] { } (E);
\path (B) edge node[below =0.15 cm] { } (F);
\path (C) edge [bend right =90] node[below =0.15 cm] { } (G);
\path (C) edge [bend left =5] node[below =0.15 cm] { } (E);
\end{tikzpicture}

\end{center}



\noindent Since the subgraph is connected, the graph $C_7$ is connected. Thus, follow the discussion at the beginning of \S \ref{closuresection}, $\mathcal{AS}_{14}$ is also connected.

\end{example}

\begin{proposition}
\label{cortype3}

Suppose $\overrightarrow{E}_1=\{e,e_3,\ldots,e_n\} \prec \overrightarrow{E}_2=\{e_1,e_2,e_3,\ldots,e_n\}$  with an edge in $G_d$ of the form $\{e\} $ to $\{e_1,e_2\}$ such that $e_1 >p$ and $\Gamma_{\overrightarrow{E}_1}$ is not in the closure of $\Gamma_{\overrightarrow{E}_2}$. Then $\Gamma_{\overrightarrow{E}_1}$ and $\Gamma_{\overrightarrow{E}_2}$ are both in the closure of $\Gamma_{{\overrightarrow{E}_3}}$, where  $\overrightarrow{E}_3=\{p-1,e_1-p+1,e_2,e_3,\ldots\,e_n\}$.

\end{proposition}

\begin{proof}

Suppose $e_i-1=pr_i+s_i$ where $1 \le s_i \le p-1$ and $r_1 \ge 1$. Then $e-1=p(r_1+r_2)+s_1+s_2+1$. By Corollary \ref{openquestion1}, $\Gamma_{\overrightarrow{E}_1}$ is not in the closure of $\Gamma_{\overrightarrow{E}_2}$ if and only if $\lfloor \frac{e-1}{p} \rfloor = \lfloor \frac{e_1-1}{p} \rfloor + \lfloor \frac{e_2-1}{p} \rfloor$, or $3 \le s_1+s_2+1 \le p-1$. Thus, $e_1-p=p(r_1-1)+s_1+1$, and $s_1+1 <p-1$. It follows that $\lfloor \frac{e_1-1}{p} \rfloor> \lfloor \frac{p-1}{p}\rfloor + \lfloor \frac{e_1-p+1-1}{p} \rfloor$. Whence, $\Gamma_{\overrightarrow{E}_2}$ is in the closure of $\Gamma_{\overrightarrow{E}_3}$ by Lemma \ref{theoremclosure} $(a)$. Also, by Lemma \ref{theoremclosure} $(c)$, $\Gamma_{\overrightarrow{E}_1}$ is in the closure of $\Gamma_{\overrightarrow{E}_3}$.


\end{proof}

The above proposition gives us a strategy to study the connectedness of $C_d$: even though there is no edge on $C_d$ directly connecting two vertices, we can still go from one vertex to another through a third one. The following corollary will show that for most choices of the genus $g$, we can connect all the vertices of $C_d$ corresponding to partitions where at least one entry has a value greater than or equal to $p$.

\begin{corollary}
\label{corhighlevel}

Fix a base field of characteristic $p>0$. 

\begin{enumerate}
    \item If $d+2 \not\equiv 1 \pmod p$, then all vertices of $C_d$ that correspond to partitions with at least one entry having a value greater than or equal to $p$ lie in the same connected component.
    \item If $d+2 \equiv 1 \pmod p$, then all vertices of $C_d$ that correspond to partitions with at least one entry having a value greater than or equal to $p$ are in the same connected component with a partition on the second level.
\end{enumerate} 

\end{corollary}

\begin{proof}
\textit{Case $d+2 \not\equiv 1 \pmod p$}. By Lemma \ref{lemmaminimalpartition}, we have that $\overrightarrow{E} \succ \{d+2\}$ for all $\overrightarrow{E} \in \Omega_d$. Thus, it suffices to show that all partitions having at least one entry greater than or equal to $p$ can be connected with $\{d+2\}$ in $C_d$.

Suppose $\overrightarrow{E}=\{e_1,\ldots,e_n\}$ with $e_1 > p$. Then there is a ``path'' from $\{d+2\}$ to $\{e_1,\ldots,e_n\}$ in $G_d$ such that each partition in the path has at least one entry greater than $p$. By Proposition \ref{cortype3}, all the partitions in the path are in the same connected component in $C_d$. In other words, one can go from $\{d+2\}$ to $\overrightarrow{E}$ in $C_d$.

If $\overrightarrow{E}=\{p,e_2, \ldots,e_n\}$ then $\lfloor \frac{p+e_2-1}{p}\rfloor=1+\lfloor \frac{e_2-1}{p}\rfloor>\lfloor \frac{p-1}{p}\rfloor + \lfloor \frac{e_2-1}{p}\rfloor=\lfloor \frac{e_2-1}{p}\rfloor$. Thus, there is an edge from $\overrightarrow{E}'=\{p+e_2, \ldots, e_n\}$ to $\overrightarrow{E}$ in $C_d$ by Lemma \ref{theoremclosure}$(a)$. By the argument in the previous paragraph, $\overrightarrow{E}'$ is in the same component with $\{d+2\}$, then so is $\overrightarrow{E}$.

\textit{Case $d+2 \equiv 1 \pmod p$}. Consider a partition $\overrightarrow{E}=\{e_1,\ldots,e_n\} \in \Omega_d$. Suppose, without loss of generality, that $e_1 >p$ and $e_n \not\equiv 0 \pmod{p}$. Then, the number $e'=\sum_{i=1}^{n-1}e_i$ is greater than $p$ and not congruent to $1$ modulo $p$. Thus, $\overrightarrow{E}=\{e_1,\ldots,e_n\} \prec \overrightarrow{E}'=\{e',e_n\}$, and by the previous discussion, $\{e'\}$ and $\{e_1, \ldots, e_{n-1}\}$ are linked in $C_{e'-2}$. Thus, $\overrightarrow{E}'$ and $\overrightarrow{E}$ lie in the same connected component in $C_d$.
\end{proof}

\begin{remark}
Corollary \ref{corhighlevel} implies that the problem of determining whether $\mathcal{AS}_g$ with $d+2 \not\equiv 1 \pmod p$ is connected can be reduced to the study of whether all strata indexed by partitions with all entries less than $p$ contain in their closure at least one stratum enumerated by a partition with at least one entry greater than or equal to $p$.
\end{remark}

\subsection{Case $p=3$}

\begin{proposition}
\label{propchar3}

The strata $\overrightarrow{E}_1, \overrightarrow{E}_2$ satisfy $\overrightarrow{E}_1 \prec \overrightarrow{E}_2$ if and only if $\Gamma_{\overrightarrow{E}_1}$ is in the closure of $\Gamma_{\overrightarrow{E}_2}$.
\end{proposition}

\begin{proof}

``If'' follows from Lemma 4.3 of \cite{MR2985514}.

Conversly, suppose $\overrightarrow{E}_1$ and $\overrightarrow{E}_2$ are connected by an edge of the form $\{e\} \rightarrow \{e_1,e_2\}$. Suppose moreover that $e_2-1=3a+u$ and $e_1-1=3b+v$ with $1 \le u,v <3$. Then $e-1=e_1+e_2-1=3(a+b)+u+v+1$. In order for $3 \nmid (e-1)$, we need $u+v+1 \ne 3$. Moreover, since $u,v \ge 1$, we must have $u+v+1>3$, thus $\lfloor \frac{e-1}{3} \rfloor > \lfloor \frac{e_1-1}{3} \rfloor + \lfloor \frac{e_2-1}{3} \rfloor$. Hence, by Lemma \ref{theoremclosure} $(a)$,  $\Gamma_{\overrightarrow{E}_1}$ is in the closure of $\Gamma_{\overrightarrow{E}_2}$. If the edge connecting $\overrightarrow{E}_1$ and $\overrightarrow{E}_2$ is of the form $\{e\} \rightarrow \{e_1,e_2,e_3\}$, we get the same result by Lemma \ref{theoremclosure} $(b)$. Thus, one can conclude that for any $\overrightarrow{E}_1 \prec \overrightarrow{E}_2$, $\Gamma_{\overrightarrow{E}_1}$ is in the closure of $\Gamma_{\overrightarrow{E}_2}$.

\end{proof}

\begin{proposition}
When the base field has characteristic $3$, the moduli space $\mathcal{AS}_g$ is connected for any value of $g$.

\end{proposition}

\begin{proof}

Just like the case $p=2$, Proposition \ref{propchar3} implies that the graph $G_d$ is a subgraph of $C_d$. Hence, in order to show that $\mathcal{AS}_g$ is connected, it suffices to prove that $G_d$ is connected, which follows from Proposition \ref{connectednessfirstgraph}.

\end{proof}

\subsection{Case $p=5$}

\begin{proposition} Given $p=5$, the graph $C_d$ is connected when $d \ge 7$. In particular, $\mathcal{AS}_g$ is connected when $g \ge 14$.

\end{proposition}

\begin{proof}

When $g \ge 14$, then $d+2 \ge 9$ so $d \ge 7$.

\bigskip

\emph{Case $d+2 \not\equiv 1 \pmod 5$}: It suffices to show that every partition of $\Omega_d$ lies in the same connected component with $\{d+2\}$ in $C_d$.

\begin{itemize}
    
\item  If $\overrightarrow{E}=\{e_1,\ldots,e_n\}$ has at least one entry greater than or equal to $5$, $\overrightarrow{E}$ is in the same component with $\{d+2\}$ by Proposition \ref{cortype3}.

\item If $\overrightarrow{E}=\{e_1,\ldots,e_n\}$ has no entry greater than $5$, then $e_i$ can only be $4,3$ or $2$. 

\noindent Suppose there exists an entry equal to $4$, say $e_1=4$. Since $d+2 \ge 9$, there exists at least one more entry, say $e_2$. Let $e'=e_1+e_2$. If $e_2=3$ or $e_2=4$, then $e'\not\equiv 1 \bmod 5$ and $\lfloor \frac{e'-1}{5} \rfloor = \lfloor \frac{e_1+e_2-1}{5} \rfloor >  \lfloor \frac{4+2-1}{5} \rfloor =1$ and $\lfloor \frac{e_1-1}{5} \rfloor +  \lfloor \frac{e_2-1}{5} \rfloor =0$ . Thus, by Lemma \ref{theoremclosure}$(a)$, $\Gamma_{\{e'\}}$ is in the closure of $\Gamma_{\{e_1,e_2\}}$, that means $\Gamma_{\{e_1,\ldots,e_n\}}$ is in the closure of $\Gamma_{\{e',\ldots,e_n\}}$ which is in the same component with $\Gamma_{\{d+2\}}$. If $e_i=2$ for all $i \ge 2$, then $n > 3$. Thus, there is an edge from $\{8,e_4,\ldots,e_n\}$ to $\{4,2,2,e_4,\ldots,e_n\}$ in $C_d$ by Lemma \ref{theoremclosure}$(c)$, and $\{8,e_4,\ldots,e_n\}$ is in the same component with $\{d+2\}$.

\noindent Suppose all entries are equal to $2$ or $3$. If there are more than three entries equal to $3$ or more than $5$ entries equal to $2$, then, by Lemma \ref{theoremclosure}$(b)$ and $(d)$, there is an edge in $C_d$ that connects $\overrightarrow{E}$ with a partition $\overrightarrow{E}'$, which is different from $\overrightarrow{E}$ by edges in $G_d$ of the form $\{9\} \rightarrow \{3,3,3\}$ or $\{10\} \rightarrow \{2,2,2,2,2\}$. So, $\overrightarrow{E}$ is in the same component with  $\{d+2\}$. Otherwise, since $d+2 \ge 9$, we are left with the only possible partitions: $\{3,2,2,2\}$ or $\{3,3,2,2\}$. These particular cases are solved by Lemma \ref{theoremclosure} $(e)$ . 

\end{itemize}

\bigskip

\emph{Case $d+2 \equiv 1 \pmod 5$}: All the minimal deformations on the second level are of the form $\{5a_1+2,5a_2+4\}$ or $\{5b_1+3,5b_2+3\}$ where $\sum_i a_i=\sum_j b_j$. It follows that:

\[\{5b_1+3,5b_2+3\} \prec \{\underbrace{3,\ldots,3}_{b_1+b_2+2},\underbrace{2,\ldots,2}_{b_1+b_2}\} \textit{ and } \{5a_1+2,5a_2+4\} \prec \{\underbrace{3,\ldots,3}_{a_1+a_2},\underbrace{2,\ldots,2}_{a_1+a_2+3}\}.\]

\noindent Thus, by an argument similar to the case $d+2 \not\equiv 1 \bmod 5$, all partitions on the branches started by ones of the form  $\{5a_1+2,5a_2+4\}$ (or $\{5b_1+3,5b_2+3\}$) are in the same components in $C_d$.

In order to show $C_d$ is connected, it now suffices to show that $\{5a_1+2,5a_2+4\}$ and $\{5b_1+3,5b_2+3\}$ are in the same component in $G_d$ for some values of $a_1,a_2,b_1$ and $b_2$. Again, by the condition that $d+2 \ge 9$, the smallest $d+2$ that can be congruent to $1$ mod $5$ is $11$. Thus, there exists in the second level a partition of the form $\{5a_1+2,5a_2+4\}$ where $a_2 \ge 1$. Whence $\{5a_1+2,5a_2+4\} \prec \{5a_1+2,5(a_2-1),9\} \prec \{3,\ldots,3,2,\ldots,2\}$, where the number of entries of the latter partition equal to $3$ is $a_1+a_2-1+3=a_1+a_2+2=b_1+b_2+2$ and the number of entries equal to $2$ is $a_1+1+a_2-1=a_1+a_2=b_1+b_2$. Thus, all partitions are in the same component in $C_d$.

\end{proof}

\subsection{Case $p>5$}

\begin{proposition}
Given $p>5$, if $d\ge \frac{p(p-1)^2}{4}$ or $d \le 1$ then $C_d$ is connected. In particular, the moduli space $\mathcal{AS}_g$ is connected when $g\ge \frac{(p^3-2p^2+p-8)(p-1)}{8}$ or $g \le \frac{p-1}{2}$.
\end{proposition}

\begin{proof}
 \emph{Case $d+2 \not\equiv 1 \bmod p$}: By Proposition \ref{cortype3}, all partitions of $\Omega_d$ with at least one entry greater than or equal to $p$ lie in the same connected component of $C_d$. Thus, it suffices to show that when $g \ge \frac{(p^3-2p^2+p-8)(p-1)}{8}$ i.e $d \ge \frac{p(p-1)^2}{4}$, every partition of the form $\{e_1,e_2,\ldots,e_n\}$ such that $e_i <p$ for all $i$ can be connected with some partition having an entry not less than $p$. 

Consider a partition of $d+2$ with all entries smaller than or equal to $\frac{p+1}{2}$. By the Pigeonhole Principle, we can guarantee to have $p$ entries with the same value $n$ between $1$ and $\frac{p-1}{2}$ when the partition has at least $\frac{p(p-1)}{2}$ entries. Thus, when $d \ge \frac{p(p-1)^2}{4}$, the partition we consider is in the same component with $\overrightarrow{E}=\{(n+1)p-n^2+1,e_2,\ldots,e_m\}$ for some $n$ between $1$ and $\frac{p-1}{2}$ by Lemma \ref{theoremclosure}$(d)$. The first entry of $\overrightarrow{E}$ is greater than $p$. Hence, $\overrightarrow{E}$ and $\{d+2\}$ are linked.

Consider a partition of $d+2$, called $\overrightarrow{E}=\{e_1,e_2, \ldots,e_n\}$, where one entry has value $m$ such that $\frac{p+1}{2}<m<p$. When $\sum_{i=1}^n e_i>\frac{p(p-1)^2}{4}$, we must have the sum of all entries except the one with value $m$ greater than $\frac{p(p-1)^2}{4}-m$. If there is another entry which has value greater than $p-m+1$, call the value $n$, then the corresponding stratum lies in the same connected component with one indexed by $\{m+n, \ldots\} \prec \{m,n, \ldots \}=\overrightarrow{E}$ by Theorem \ref{theoremclosure}(a). Otherwise, all other entries are smaller than or equal to $p-m+1$. Again by the Pigeonhole Principle, we can guarantee to have $p$ entries with the same values $n$ such that $1\le n\le p-m+1$ when the partitions have at least $p(p-m+1)+1$ entries (``$+1$'' coming from the entry with value $m$). Thus, if $\frac{p(p-1)^2}{4}-m< p(p-m+1)^2$, by a similar argument as in the previous paragraph,  $\overrightarrow{E}$ and $\{d+2\}$ are in the same connected component. In order to prove the above claim, we study the function:

\[ F(m):=p(p-m+1)^2-\dfrac{p(p-1)^2}{4}+m.\]




\noindent One can show that $F(m)$ is strictly increasing on the closed interval $[\frac{p+1}{2}, p-1]$. Moreover, the value of $F(m)$ at $(p+1)/2$ is:  $F(\frac{p+1}{2})=\frac{p+1}{2}>0$. Hence, $F(m)$ is positive on the whole interval $[\frac{p+1}{2}, p-1]$ and we are done.

\emph{Case $d+2 \equiv 1 \bmod p$}: All deformations on the second level are of the form $\{a+pe_1,p+1-a+pe_2\}$ where $2\le a \le p-1$. Suppose $d+2$ is odd. Then one of the entries of each partition in the second level is even and the another is odd. Thus, we have $\{a+pe_1,p+1-a+pe_2\} \prec \{2,2,\ldots,2,3\}$ for all possible $a$, $e_1$, $e_2$. Also, the number of entries equal to $2$ in the latter partition is greater than $p$ and one of the entries of the former partition is greater than $2p$ by the assumption that $d+2 \ge  \frac{p(p-1)^2}{4}$. We have a chain of partitions $\{a+pe_1,p+1-a+pe_2\} \prec \{2p,e+(e_1-2)p,p+1-a+pe_2\}$ (or $\{2p,e+e_1 p,1-a+p(e_2-1)\}$) $\prec \{2p,2,\ldots,2,3\} \prec \{2,2,\ldots,2,3\}$ where the first three are in the same component of $C_d$ by Corollary \ref{corhighlevel}. Applying Lemma \ref{theoremclosure}$(d)$ for $n=1$, we know that there is a path in $C_d$ connecting $\{2p,2,\ldots,2,3\}$ and $\{2,2,\ldots,2,3\}$. Furthermore, $\{a+pe_1,p+1-a+pe_2\}$ and $\{2p,2,\ldots,2,3\}$ are linked by Corollary \ref{corhighlevel}. Hence, the partition $\{a+pe_1,p+1-a+pe_2\}$ lies in the same connected component with $\{2,2,\ldots,2,3\}$. Similarly, if $d+2$ is even, $\{a+pe_1,p+1-a+pe_2\}$ can be linked with $\{2,2,\ldots,2,3,3\}$ in $C_d$ and the proof follows as in the odd case.

When $d \le 1$, i.e. when $d=1$ or $d=0$, then $\Omega_d$ has exactly one partition: either $\{3\}$ or $\{2\}$. Thus, the graph $C_d$ is connected.

\end{proof}

\subsection{Disconnectedness of $AS_g$}
\label{disconnectedsection}

\begin{proposition}
\label{propdisconnected}
Suppose $p \ge 5$. Then $\mathcal{AS}_g$ is disconnected when $\frac{p-1}{2} < g \le \frac{(p-1)^2}{2}$.

\end{proposition}

\begin{proof}
If $g \le \frac{(p-1)^2}{2}$, then $d+2\le p+1$. This implies that all elements of $\Omega_d$ have entries with values less than $p$. By Corollary 3.16 of \cite{MR2223481}, each irreducible stratum of $\mathcal{AS}_g$ has dimension equal to $d-1$. Thus, no closure of one stratum contains another. Hence, by Corollary \ref{corclosureintersect}, each stratum is a connected component of $\mathcal{AS}_g$. Whence, if $g$ satisfies the second inequality and $\Omega_d$ has more than one partition in $\Omega_d$ then $\mathcal{AS}_g$ is disconnected. By Corollary \ref{corollaryirreducible}, $\mathcal{AS}_g$ is irreducible if and only if $g=0$ or $g=\frac{p-1}{2}$, which implies the first inequality.

\end{proof}

\begin{remark}
The only obstructions for a deformation that have been known thus far were the dimensions of the strata and the ``$\prec$'' relation of the partitions which says $\overrightarrow{E}_1 \prec \overrightarrow{E}_2$ if $\overrightarrow{E}_2$ is a refinement of $\overrightarrow{E}_1$. More precisely, there is a deformation from  $\Gamma_{\overrightarrow{E}_1}$ to $\Gamma_{\overrightarrow{E}_2}$ only if $\dim_k(\Gamma_{\overrightarrow{E}_1})< \dim_k(\Gamma_{\overrightarrow{E}_2})$ and $\overrightarrow{E}_1 \prec \overrightarrow{E}_2$. Our new approach, which was mentioned in Remark \ref{char5}, gives us another obstruction and improves Proposition \ref{propdisconnected}. In particular, we are able to show that for $p>5$, $\mathcal{AS}_g$ disconnected if $(p-1)/2<g \le 2(p-1)$.
\end{remark}

 \section{Proof of Technical Results}
 \label{technical}

 In this section, we give the proof of Theorem \ref{deformationtheorem}. Each subsection will prove one item of this theorem.
 
By the discussion at the beginning of \S \ref{closure}, it suffices to study relations between two partitions $\overrightarrow{E}_1=\{e\} \prec \overrightarrow{E}_2=\{e_1,e_2, \ldots, e_n\}$.

\subsection{Edge of Type $\{e\} \rightarrow \{e_1,e_2\}$}

Throughout this subsection, suppose $\overrightarrow{E}_1=\{e\} \prec \overrightarrow{E}_2=\{e_1,e_2\}$. Then there is an edge in $G_d$ from $\overrightarrow{E}_1$ to $\overrightarrow{E}_2$. Given $e=e_1+e_2$ such that $p \nmid (e-1)(e_1-1)(e_2-1)$. Without loss of generality, one can assume that $e_1 \ge e_2$. Let $pw$ be the smallest multiple of $p$ that is greater than or equal to $e_1$. Before going into the technical details, let us take a look at a concrete example.

\begin{example}
Let $p=5$, $g=10$ and consider $\overrightarrow{E}_1=\{7\}$ and $\overrightarrow{E}_2=\{4,3\}$. Then there is a deformation over $S=\spec k[[t]]$ given by the normalization of $\mathbb{P}^1_S$ over the extension of its fraction field that is defined by the following affine equation:

\[y^5-y=\dfrac{-2x+t}{(-2)x^5(x-t)^2}=H(x,t).\]

The special fiber has the affine equation:

\[y^5-y=\dfrac{1}{x^6},\]

\noindent which represents a $k$-point of $\Gamma_{\overrightarrow{E}_1}$.

On the generic fiber, when $t \neq 0$, the partial fraction decomposition of $H(x,t)$ is of the form:

\[\dfrac{-1}{2tx^5}+\dfrac{1}{2t^3x^3}+\dfrac{1}{t^4x^2}+\dfrac{3}{2t^5x}+\dfrac{1}{2t^4(x-t)^2}-\dfrac{3}{2t^5(x-t)}\]

\noindent One can obtain a $\mathbb{Z}/p$-curve that is isomorphic to the generic fiber by adding $\frac{1}{2tx^5}-\frac{1}{\sqrt[5]{2t}x}$ to $H(x)$. The curve is thus defined by the affine equation:

\[y^5-y=\dfrac{-1}{\sqrt[5]{2t}x}+\dfrac{1}{2t^3x^3}+\dfrac{1}{t^4x^2}+\dfrac{3}{2t^5x}+\dfrac{1}{2t^4(x-t)^2}-\dfrac{3}{2t^5(x-t)}.\]

\noindent Hence, the generic fiber is branched at two points $x=0$ and $x=t$, which have ramification jumps $3$ and $2$, respectively. Thus, it represents a $k((t))$-point of $\Gamma_{\overrightarrow{E}_2}$. 

\end{example}

\begin{remark}
The previous example also illustrates our main approach in the whole chapter. We will look for a rational function $H(x,t) \in \Frac(k[x,t])$  such that $H(x,0)=c/x^{e-1}$ and has a desired partial fraction decomposition in terms of $x$. 
\end{remark}

\begin{lemma}
\label{subprop}

Suppose $v \le u$. Then:

\[\sum_{k=0}^v {u \choose k}(-1)^k =\dfrac{(-1)^v}{v!} \prod_{l=1}^v (u-l). \]

\end{lemma}

\begin{proof}

Induction on $v$.

\noindent For $v=1$:  

\[ \sum_{k=0}^1 {u \choose k}(-1)^k ={u \choose 0} -{u \choose 1}=1-u=\dfrac{(-1)^1}{1!} \prod_{l=1}^1 (u-l).\]

\noindent Suppose the proposition is true for $v=n$. Then, for $v=n+1$:

\begin{equation} \label{induction}
\begin{split}
\sum_{k=0}^{n+1} {u \choose k}(-1)^k & =\sum_{k=0}^n {u \choose k}(-1)^k + \dfrac{(-1)^{n+1}}{(n+1)!} \prod_{m=0}^{n} (u-m) \\
&=\dfrac{(-1)^n}{n!} \prod_{l=1}^n (u-l)+ \dfrac{(-1)^{n+1}}{(n+1)!} \prod_{m=0}^{n} (u-m)  \hspace{3mm}(\text{by induction}) \\
 & =\dfrac{(-1)^{n+1}}{(n+1)!} \prod_{l=1}^{n+1} (u-l) . 
\end{split}
\end{equation}

\end{proof}

\begin{lemma}
\label{type1lemma1}

Suppose $\lfloor \frac{e-1}{p} \rfloor > \lfloor \frac{e_1-1}{p} \rfloor + \lfloor \frac{e_2-1}{p} \rfloor$. Then $p \nmid {{e_2-1}\choose {pw-e_1}}$,  $p \nmid {{e_2-1}\choose {pw-e_1+1}}$ and $p \nmid \sum_{k=0}^{pw-e_1}{{e_2-1} \choose k} (-1)^{e_2-1-k}=:Z$.

\end{lemma}

\begin{proof}

Suppose $e_2-1=pa+u$, and $e_1-1=pb+v$ such that  $1 \le u,v <p$ and either $a \le b$ or $u \le v$ if $a=b$. Thus, $e-1=e_1+e_2-1=p(a+b)+u+v+1$, and $pw=p(b+1)$. Because $\lfloor \frac{e-1}{p} \rfloor > \lfloor \frac{e_1-1}{p} \rfloor + \lfloor \frac{e_2-1}{p} \rfloor$, we have $p+1 \le u+v+1 \le 2p-1$.

Now consider:

\[{{e_2-1}\choose {pw-e_1}}=\dfrac{(e_2-1)!}{(pw-e_1)!(e_1+e_2-1-pw)!}=\dfrac{1}{(pw-e_1)!} (e_1+e_2-pw) \ldots (e_2-1).\]

\noindent Because $e_1+e_2-pw=pa+u+v-p+2 \ge pa+2$, we have $p \nmid   (e_1+e_2-pw) \ldots (e_2-1)$. Thus, $p \nmid {{e_2-1}\choose {pw-e_1}}$.

\bigskip

Similarly,

\[{{e_2-1}\choose {pw-e_1+1}}=\dfrac{(e_2-1)!}{(pw-e_1+1)!(e_1+e_2-2-pw)!}=\dfrac{1}{(pw-e_1+1)!} (e_1+e_2-1-pw) \ldots (e_2-1).\]
 
\noindent Because  $e_1+e_2-1-pw=pa+u+v-p+1 \ge pa+1$, we have $p \nmid   (e_1+e_2-1-pw) \ldots (e_2-1)$. Thus, $p \nmid {{e_2-1}\choose {pw-e_1+1}}$.

\bigskip

Now we want to show $p \nmid Z=\sum_{k=0}^{pw-e_1} \dfrac{(e_2-k)\ldots (e_2-1)}{1\cdot 2 \ldots k} (-1)^{e_2-1-k}$. Taking reduction modulo $p$, it suffices to prove that:

\[p \nmid \sum_{k=0}^{p-v-1} {u \choose k} (-1)^k \textrm{ for } 1 \le u,v <p , u+v \ge p.\]

\noindent Apply Lemma \ref{subprop}, $\sum_{k=0}^{p-v-1} {u \choose k} (-1)^k=\dfrac{(-1)^{p-v-1}}{(p-v-1)!} \prod_{l=1}^{p-v-1} (u-l) $. Since $p \nmid (u-l)$ for any $1 \le l \le (p-v-1)$, we conclude that $p\nmid \sum_{k=0}^{p-v-1} {u \choose k} (-1)^k$.

\end{proof}

\begin{proposition}
\label{type1lemma2}

Suppose $\lfloor \frac{e-1}{p} \rfloor > \lfloor \frac{e_1-1}{p} \rfloor + \lfloor \frac{e_2-1}{p} \rfloor$. Then there exists a rational function $G(x,t) \in \Frac(k[x,t])$ that has partial fraction decomposition in $k(t)(x)$ of the form:

\[ G(x,t)= \sum_{i=1}^{e_1-1} \dfrac{a_i}{x^i} + \dfrac{a_{pw}}{x^{pw}} + \sum_{j=1}^{e_2-1} \dfrac{b_j}{(x-t)^j}\]

\noindent such that $a_{pw}, a_{e_1-1}, b_{e_2-1} \ne 0$; $a_i$,$b_j$ are elements of $k(t)$; and $G(x,0) \in \Frac(k[x])$ has exactly one pole of order $e-1$ at $x=0$.

\end{proposition}

\begin{proof}

If $G(x,t)$ is the rational function we are looking for, it should have the fraction form:

\[ G(x,t)= \dfrac{ \sum_{i=1}^{e_1-1} a_i x^{pw-i} (x-t)^{e_2-1} + a_{pw}(x-t)^{e_2-1}+ \sum_{j=1}^{e_2-1} b_j x^{pw} (x-t)^{e_2-1-j}}{x^{pw}(x-t)^{e_2-1}}=\dfrac{N(x,t)}{D(x,t)}.\]


\noindent We want the numerator of $G(x,t)$, called $N(x,t)$, to lie in $k[t][x]$ and $N(x,0)$ to be a $k^{\times}$- multiple of $x^{pw-e_1}$.

In order to construct such a function, it suffices to find values for $a_{pw}$, $a_1$, $\ldots$,$a_{e_1-1}$,$b_1$, $\ldots, b_{e_2-1}$ in $k(t)$ that satisfy the following conditions:

\begin{enumerate}
    \item Each term of $N(x,t)$ having power of $x$ less than $pw-e_1$ is a multiple of $t$.
    \item The term with power of $x$ equal to $pw-e_1$ is not a multiple of $t$.
    \item All the terms of $N(x,t)$ having $x$-degree greater than $pw-e_1$ are equal to $0$.
    \item $a_{pw}$, $a_{e_1-1}$ and $b_{e_2-1}$ are not equal to $0$.
\end{enumerate}


Observe that all the terms of $ b_j x^{pw} (x-t)^{e_2-1-j}$ have $x$-degree at least $pw$, which is strictly greater than $pw-e_1$. The terms of $a_i x^{pw-i} (x-t)^{e_2-1}$ have $x$-degree at least $pw-e_1+1$, which is also greater $pw-e_1$. Thus, all the terms of $N(x,t)$ with $x$-degree less than or equal to $pw-e_1$ lie in $a_{pw} (x-t)^{e_2-1}$. Note that $e_2-1 > pw-e_1$ follows from the assumption that  $\lfloor \frac{e-1}{p} \rfloor > \lfloor \frac{e_1-1}{p} \rfloor + \lfloor \frac{e_2-1}{p} \rfloor$.


The term of $N(x,t)$ with $x$-degree $n$, for $pw-e_1+1 \le n \le pw+e_2-2$, has coefficient of the form:

\[\sum_{\substack{pw-i+r=n \\ 1 \le i \le e_1-1 \\ 0 \le r \le e_2-1 }} a_i {{e_2-1} \choose {r}} (-t)^{e_2-1-r}+ \sum_{\substack{pw+l=n \\ 0 \le l \le e_2-1-j \\ 1 \le j \le e_2-1 }} b_j {{e_2-1-j} \choose {l}} (-t)^{e_2-1-j-r}+a_{pw} {{e_2-1} \choose n}(-t)^{e_2-1-n}\]

\noindent We want all these terms to be equal to $0$, i.e., $a_1, \ldots, a_{e_1-1},b_1, \ldots, b_{e_2-1}$ are the solutions of the linear system:

\[ \left( \begin{array}{cc}\

A_{ij}

\end{array} \right) \times \left( \begin{array}{cc}
b_1 \\
b_2 \\
\vdots \\
b_{e_2-2}\\
b_{e_2-1}\\
a_1 \\
a_2 \\
\vdots \\
a_{e_1-2} \\
a_{e_1-1}

\end{array} \right) = \left( \begin{array}{cc}
0 \\
0\\ 
\vdots \\
0\\
-a_{pw}\\
-a_{pw}(-t)\\
\vdots \\
\vdots \\
-a_{pw} {{e_2-1} \choose {pw-e_1+2}}(-t)^{e_2-1-(pw-e_1+2)} \\
-a_{pw} {{e_2-1} \choose {pw-e_1+1}}(-t)^{e_2-1-(pw-e_1+1)} 
\end{array} \right)\] 

Where $A_{ij}$ is a size $(e_1+e_2-2)$ square matrix of the form: 

\bigskip

\[
\renewcommand\arraystretch{1.3}
\mleft[
\begin{array}{ccccc|ccccc}
 1&  0&  \ldots&  0&  0&  1&  0&  \ldots&  0&  0 \\ 
 (e_2-2)(-t)&  1&  \ldots&  \ldots&  \ldots&  \ldots&  \ldots&  \ldots&  \ldots&  \ldots \\ 
 \ldots&  \ldots&  \ldots&  \ldots&  \ldots&  \ldots&  \ldots&  \ldots&  \ldots&  \ldots \\ 
 \ldots&  \ldots&  \ldots&  1&  0&  \ldots&  \ldots&  \ldots&  \ldots&  \ldots \\ 
 \ldots&  \ldots&  \ldots&  \ldots&  1&  \ldots&  \ldots&  \ldots&  \ldots&  \ldots\\ 
 \hline
 0&  0&  \ldots&  0&  0&  (-t)^{e_2-1}&  \ldots&  \ldots&  \ldots&  \ldots \\ 
 0&  0&  \ldots&  0&  0&  0&  (-t)^{e_2-1}&  \ldots&  \ldots&  \ldots \\ 
 \ldots&  \ldots&  \ldots&  \ldots&  \ldots&  \ldots&  \ldots&  \ldots&  \ldots&  \ldots \\ 
 0&  0&  \ldots&  0&  0&  0&  0&  \ldots&  (-t)^{e_2-1}&  \ldots\\ 
 0&  0&  \ldots&  0&  0&  0&  0&  \ldots&  0& (-t)^{e_2-1} 
\end{array}
\mright]
\]

\noindent The entry $a_{ij}$ of the matrix is the multiple of $b_j$ if $1 \le j \le e_2-1$ or the multiple of $a_{j-e_2+1}$ if $e_2 \le j \le e_1+e_2-2$ in the coefficient of the term with $x$-degree $pw+e_2-1-i$. The right hand side of the linear system is a $(e_1+e_2-2) \times 1$ matrix where the $i$-th row is the coefficient of $x$-degree $pw+e_2-1-i$ term of the polynomial $-a_{pw}(x-t)^{e_2-1}$. One can easily see that $A_{ij}$ is a block triangular matrix. The size of the blocks, clockwise from top, are $(e_2-1)\times (e_2-1)$, $(e_1-1)\times (e_1-1)$, $(e_2-1)\times (e_2-1)$, and $(e_1-1)\times (e_1-1)$. Thus, it is invertible and we can always find the desired values for $a_{pw}$, $a_1$, $\ldots$,$a_{e_1-1}$,$b_1$, $\ldots, b_{e_2-1}$.

Now consider:

\[ a_{pw} (x-t)^{e_2-1}= \sum_{r=0}^{e_2-1} a_{pw} {{e_2-1} \choose r} x^r (-t)^{e_2-1-r}. \]

\noindent Pick $a_{pw}=\dfrac{1}{t^{e_1+e_2-1-pw}}.$ Then  $G(x,t)$ has the form:

\begin{equation}
\label{type1edge}
    G(x,t)= \dfrac{ \sum_{r=0}^{pw-e_1} {{e_2-1} \choose r}x^r (-1)^{e_2-1-r} t^{pw-e_1-r}}{x^{pw}(x-t)^{e_2-1}}. 
\end{equation}

\noindent With that choice of $a_{pw}$, conditions $1, 2, 3$ and the part of condition $4$ that $a_{pw} \neq 0$ are satisfied.

 It now suffices to show that $ a_{e_1-1} \neq 0$ and $b_{e_2-1} \neq 0$. Consider the coefficient of the term of $N(x,t)$ with $x$-degree $pw-e_1+1$, which is equal to $0$ by the previous construction:

$$a_{e_1-1} (-t)^{e_2-1} + a_{pw} {{e_2-1} \choose {pw-e_1+1}}(-t)^{e_2-1-(pw-e_1+1)}=a_{e_1-1} (-t)^{e_2-1} +{{e_2-1} \choose {pw-e_1+1}} (-1)^{e-1-pw} \dfrac{1}{t}.$$

\noindent By Lemma \ref{type1lemma1}, $p \nmid {{e_2-1} \choose {pw-e_1+1}}.$ So, in order for the  $pw-e_1+1^{th}$ coefficient to be $0$, we must have $a_{e_1-1} \ne 0.$ In particular, we have $a_{e_1-1}={{e_2-1} \choose {pw-e_1+1}} (-1)^{e_1-e_2+1-pw} t^{-e_2}$.


Now, we want to prove that $b_{e_2-1} \ne 0$. Consider:

\[H(x,t)=G(x,t)(x-t)^{e_2-1}=\dfrac{ \sum_{r=0}^{pw-e_1} {{e_2-1} \choose r}x^r (-1)^{e_2-1-r} t^{pw-e_1-r}}{x^{pw}}.\]

\noindent The coefficient of the leading term of the expansion of $H(x,t)$ around $x=t$ coincides with the coefficient of the leading term of the expansion of $G(x,t)$ around $x=t$, i.e., $b_{e_2-1}$. Note that:

\[H(t,t)= \dfrac{\sum_{r=0}^{pw-e_1}{{e_2-1} \choose r}  (-1)^{e_2-1-r}}{t^{e_1}}=b_{e_2-1}.\]
\noindent By Lemma \ref{type1lemma1}, $p \nmid \sum_{r=0}^{pw-e_1}{{e_2-1} \choose r} (-1)^{e_2-1-r}$. So $b_{e_2-1} \ne 0$.

\end{proof}

The following proposition will prove part $(a)$ of Theorem \ref{deformationtheorem}.

\begin{proposition}
\label{proptype1}

Suppose $\lfloor \frac{e-1}{p} \rfloor > \lfloor \frac{e_1-1}{p} \rfloor + \lfloor \frac{e_2-1}{p} \rfloor$ . Suppose moreover that $\psi_0$ is an Artin-Schreier cover over $k$, branched at a point $b$ with ramification jump $e_1+e_2-1$. Then there exists an Artin-Schreier cover $\psi_S$ over $S=\spec k[[t]]$ whose special fiber is isomorphic to $\psi_0$, whose generic fiber is branched at two points that specialize to b and which have ramification jumps $e_1-1$ and $e_2-1$, and which is unramified everywhere else.
\end{proposition}

\begin{proof}

We use notation from Lemma \ref{type1lemma1}. 

By Proposition \ref{localAS} and \ref{propreduce}, it suffices to construct a deformation for the germ of $\psi_0$ at $b$.

Let us study the Artin-Schreier cover $\Psi$ over $S=\spec k[[t]]$ given by the normalization of $\mathbb{P}^1_S$ over the extension of its fraction field that is defined by the following affine equation:

\begin{equation}
\label{type1affine}    
y^p-y =\dfrac{ \sum_{r=0}^{pw-e_1} {{e_2-1} \choose r}x^r (-1)^{e_2-1-r} t^{pw-e_1-r}}{{{e_2-1} \choose {pw-e_1}}(-1)^{e_1+e_2-1-pw}x^{pw}(x-t)^{e_2-1}}:=F(x,t) 
\end{equation}

\noindent where the right-hand-side of (\ref{type1affine}) is just a scalar multiple of (\ref{type1edge}). We will show that $\Psi$ is locally a deformation of $\psi_0$ that satisfies the conditions we seek.


On the special fiber of $\Psi$, when $t=0$, we have:

\[F(x,0)=\dfrac{{{e_2-1} \choose {pw-e_1}}(-1)^{e_1+e_2-1-pw}}{{{e_2-1} \choose {pw-e_1}}(-1)^{e_1+e_2-1-pw}x^{e-1}}=\dfrac{1}{x^{e-1}}.\]

\noindent Thus, the normalization of the special fiber $\Psi_s$ of $\Psi$ has ramification jump $e-1$ at its branch point. Hence, $\Psi_s$ is birationally equivalent to $\psi_0$.

On the generic fiber, when $t \ne 0$, the cover $\Psi_{\eta}$ is branched above $x=0$ and $x=t$. The order of the pole of $F(x,t)$ at $x=t$ is $e_2-1$ by the construction from Proposition \ref{type1lemma2}. Since, $e_2-1$ is prime to $p$ by hypothesis, the ramification jump above $x=t$ is $e_2-1$. 

The expansion of $F(x,t)$ around $x=0$ is of the form:

\[ \dfrac{a_{pw}}{x^{pw}}+\sum_{i=1}^{e_1-1} \dfrac{a_i}{x^i} +  \textrm{terms with non negative $x$-degree,} \]

\noindent Where $a_{pw}=\dfrac{1}{{{e_2-1} \choose {pw-e_1}}(-1)^{e_1+e_2-1-pw}t^{e_1+e_2-1-pw}}$. After a finite inseparable extension of $k((t))$ with equation $t_1 ^p =t$, the leading term of $F(x,t)$ is a $p$-th power. The second term is $a_{e_1-1}={{e_2-1} \choose {pw-e_1+1}} (-1)^{e_1-e_2+1-pw} t^{-e_2}$, which is non-zero and thus becomes the leading term of the affine equation in standard form for $\Psi_{\eta}$. Thus, the ramification jump above $x=0$ is $e_1-1$.

Hence, the cover $\Psi_{\eta}$ is branched at two points that specialize to $0$ which have ramification jumps $e_1-1$ and $e_2-1$. By Lemma \ref{lemmagenus} and Lemma IV.2.3 of \cite{MR1011987}, $\Psi$ is a deformation of $\psi_0$ over $S$.

\end{proof}

\begin{remark}
Suppose $p \mid e_1 e_2$. Without loss of generality, we may assume $p \mid e_1$. Then, we can use $pw=e_1$ and equation (\ref{type1affine}) becomes:

\begin{equation}
\label{type1PriesZhu}
y^p-y=\dfrac{1}{x^{e_1}(x-t)^{e_2-1}}.    
\end{equation}

\noindent Which is precisely equation (\ref{RPdeform}) after a change of variables.

\end{remark}

\subsection{Edge of Type $\{e\} \rightarrow \{e_1,e_2,e_3\}$}

From the discussion in the beginning of section \ref{partition}, we know that there is an edge of type $\{e\} \rightarrow \{e_1,e_2,e_3\}$ in $G_d$ if and only if $e_i \equiv \frac{p+1}{2} \pmod p$. Let $pw=e_1+\frac{p-1}{2}$ be the smallest multiple of $p$ greater than or equal to $e_1$.

The proof of the below proposition is easy, and we will leave it for the reader.

\begin{lemma}
\label{funlemma}
Suppose $i<p$. Then ${a \choose i} \equiv {{\overline{a}} \choose i} \pmod p$, for $0 \le \overline{a} <p$, $a \equiv \overline{a} \pmod p$.

\end{lemma}




\begin{lemma}
\label{lemmakeith}

 The polynomial $\sum_{i=0}^{\frac{p-1}{2}} {{\frac{p-1}{2}}\choose {i}}^2  x^{\frac{p-1}{2}-i}$ is separable over $k[x]$ where $k$ has characteristic $p$.

\end{lemma}

\begin{proof}
Theorem 4.1 in chapter V of \cite{MR817210}.
\end{proof}

\begin{proposition}
\label{proptype2}

Suppose $\psi_0$ is an Artin-Schreier cover over $k$, branched at a point $b$ with ramification jump $e_1+e_2+e_3-1$. Then there exists an Artin-Schreier cover $\Psi$ over $S=\spec k[[t]]$ whose special fiber is isomorphic to $\psi_0$, whose generic fiber is branched at three points that specialize to b and which have ramification jumps $e_1-1$, $e_2-1$ and $e_3-1$, where $e_i \equiv \frac{p+1}{2} \pmod p$, and whose ramification divisor is otherwise constant.

\end{proposition}

\begin{proof}

Let us study the Artin-Schreier cover $\Psi$ over $S=\spec k[[t]]$ given by the normalization of $\mathbb{P}^1_S$ over the extension of its fraction field that is defined by the following affine equation:

\[ y^p-y =\dfrac{\sum_{n=0}^{\frac{p-3}{2}}( \sum_{i+j=n} {{\frac{p-1}{2}}\choose i}{{\frac{p-1}{2}} \choose j} r^{e_3-1-j}(-1)^{\frac{p-3}{2}-n})t^{\frac{p-3}{2}-n}x^n-tx^{\frac{p-1}{2}}}{(\sum_{i+j=\frac{p-3}{2}} {{\frac{p-1}{2}}\choose i}{{\frac{p-1}{2}} \choose j} r_0^{e_3-1-j})x^{pw}(x-t)^{e_2-1}(x-tr)^{e_3-1}}=:F(x,t), \]


\noindent where $r$ is a solution of:

\begin{equation}
\label{definer}
    \sum_{i+j=\frac{p-1}{2}} {{\frac{p-1}{2}}\choose {i}} {{\frac{p-1}{2}}\choose {j}} r^{e_3-1-j} =t^2.
\end{equation}

\noindent over a fixed algebraic closure of $k((t))$ such that its value when $t=0$, called $r_0$, satisfies $\sum_{i+j=\frac{p-3}{2}} {{\frac{p-1}{2}}\choose {i}} {{\frac{p-1}{2}}\choose {j}} r_0^{e_3-1-j} \ne 0$. We will show that $\Psi$ is locally a deformation of $\psi_0$ that satisfies the conditions we seek.



Since the left-hand-side of (\ref{definer}) is congruent to $\sum_{i=0}^{\frac{p-1}{2}} {{\frac{p-1}{2}}\choose {i}}^2  r^{e_3-1-i}$ modulo $p$ by Lemma \ref{funlemma}, showing that there exists $r \in \overline{k((t))}$ which satisfies the above conditions is equivalent to showing that there are some roots of the polynomial $\sum_{i=0}^{\frac{p-1}{2}} {{\frac{p-1}{2}}\choose {i}}^2  r^{\frac{p-1}{2}-i}$ that are not roots of the polynomial $\sum_{i+j=\frac{p-3}{2}} {{\frac{p-1}{2}}\choose {i}} {{\frac{p-1}{2}}\choose {j}} r^{e_3-1-j}=r^{e_3-1-\frac{p-3}{2}}(\sum_{i+j=\frac{p-3}{2}} {{\frac{p-1}{2}}\choose {i}} {{\frac{p-1}{2}}\choose {j}} r^{\frac{p-3}{2}-j})$. Since the former polynomial has degree higher than $\sum_{i+j=\frac{p-3}{2}} {{\frac{p-1}{2}}\choose {i}} {{\frac{p-1}{2}}\choose {j}} r^{\frac{p-3}{2}-j}$, it suffices to show the former one is separable, which follows from Lemma \ref{lemmakeith}.

On the special fiber, when $t=0$, then:

\[F(x,0)=\dfrac{x^{\frac{p-3}{2}}  }{x^{pw}(x-0)^{e_2-1}(x-0)^{e_3-1}}= \dfrac{1}{x^{e-1}}.\]

\noindent Thus, the normalization of the special fiber $\Psi_s$ has ramification jump $e-1$ at its only branch point. Hence, the cover $\Psi_s$ is birationally equivalent to $\psi_0$.

On the generic fiber, when $t \ne 0$, then $\Psi_{\eta}$ is branched above $x=0$, $x=t$ and $x=tr$. Consider:

\[H(x,t)=F(x,t)(x-t)^{e_2-1}=\dfrac{\sum_{n=0}^{\frac{p-3}{2}} \sum_{i+j=n} {{\frac{p-1}{2}}\choose i}{{\frac{p-1}{2}} \choose j} r^{e_3-1-j}(-1)^{\frac{p-3}{2}-n}t^{\frac{p-3}{2}-n}x^n-tx^{\frac{p-1}{2}}}{(\sum_{i+j=\frac{p-3}{2}} {{\frac{p-1}{2}}\choose i}{{\frac{p-1}{2}} \choose j} r_0^{e_3-1-j})x^{pw}(x-tr)^{e_3-1}}. \]

\noindent Then:

\[H(t,t)=\dfrac{\sum_{n=0}^{\frac{p-3}{2}} \sum_{i+j=n} {{\frac{p-1}{2}}\choose i}{{\frac{p-1}{2}} \choose j} r^{e_3-1-j}(-1)^{\frac{p-3}{2}-n}t^{\frac{p-3}{2}}-t^{\frac{p+1}{2}}}{(\sum_{i+j=\frac{p-3}{2}} {{\frac{p-1}{2}}\choose i}{{\frac{p-1}{2}} \choose j} r_0^{e_3-1-j})t^{pw}(t-tr)^{e_3-1}} \ne 0. \]

\noindent Because $H(x,t)$ does not have a zero at $x=t$, the order of the pole of $F(x,t)$ at $x=t$ is $e_2-1$. Moreover, $e_2-1$ prime to $p$ by hypothesis. Thus, the ramification jump above $x=t$ is $e_2-1$.

Similarly, consider:

\[K(x,t)=F(x,t)(x-tr)^{e_3-1}=\dfrac{\sum_{n=0}^{\frac{p-3}{2}} \sum_{i+j=n} {{\frac{p-1}{2}}\choose i}{{\frac{p-1}{2}} \choose j} r^{e_3-1-j}(-1)^{\frac{p-3}{2}-n}t^{\frac{p-3}{2}-n}x^n-tx^{\frac{p-1}{2}}}{(\sum_{i+j=\frac{p-3}{2}} {{\frac{p-1}{2}}\choose i}{{\frac{p-1}{2}} \choose j} r_0^{e_3-1-j})x^{pw}(x-t)^{e_2-1}}.\]

\noindent Then:

\[K(tr,t)=\dfrac{\sum_{n=0}^{\frac{p-3}{2}} \sum_{i+j=n} {{\frac{p-1}{2}}\choose i}{{\frac{p-1}{2}} \choose j} r^{e_3-1-j}(-1)^{\frac{p-3}{2}-n}t^{\frac{p-3}{2}}r^n-t^{\frac{p+1}{2}}r^{\frac{p-1}{2}}}{(\sum_{i+j=\frac{p-3}{2}} {{\frac{p-1}{2}}\choose i}{{\frac{p-1}{2}} \choose j} r_0^{e_3-1-j})(tr)^{pw}(tr-t)^{e_2-1}} \ne 0.\]

\noindent By the same argument as above, the ramification jump above $x=tr$ is $e_3-1$.

Finally, let us consider the expansion of $F(x,t)$ around $x=0$:

\begin{equation}
\label{type2around0}
    \sum_{i=1}^{pw} \dfrac{a_i}{x^i} +  \textrm{terms with non negative $x$-degree,}
\end{equation}


\noindent where $a_{pw}=\frac{1}{(\sum_{i+j=\frac{p-3}{2}} {{\frac{p-1}{2}}\choose i}{{\frac{p-1}{2}} \choose j} r_0^{e_3-1-j})(-t)^{e_2+e_3-2-\frac{p-3}{2}}}$. In order for $\frac{a_{pw}}{x^{pw}}$ to have the same denominator as $F(x,t)$, we multiply the top and the bottom of the fraction by ${(x-t)^{e_2-1}(x-tr)^{e_3-1}}$. It would have the form:

\begin{equation}
\label{newapw}
    \frac{(x-t)^{e_2-1}(x-tr)^{e_3-1}(-t)^{-(e_2+e_3-2-\frac{p-3}{2})}}{(\sum_{i+j=\frac{p-3}{2}} {{\frac{p-1}{2}}\choose i}{{\frac{p-1}{2}} \choose j} r_0^{e_3-1-j})(x-t)^{e_2-1}(x-tr)^{e_3-1}}.
\end{equation}

\noindent One can check that the coefficients of the terms of the numerator of the above fraction with $x$-degree $0,1, \ldots, \frac{p-1}{2}$ are precisely the corresponding ones of the numerator of $F(x,t)$. Thus, we have $a_{pw-1}=\ldots=a_{pw-\frac{p-1}{2}}=0$. Moreover, we must have $a_{pw-\frac{p+1}{2}}=a_{e_1-1}$ different from $0$, or else we can not cancel out the term with $x$-degree $\frac{p+1}{2}$ of the numerator of (\ref{newapw}). After a finite inseparable extension of $k((t))$ with equation $t_1 ^p =t$, the leading term $\frac{a_{pw}}{x^{pw}}$ of the expansion of $F(x,t)$ around $x=0$ is a $p$-th power. The second term is $\frac{a_{e_1-1}}{x^{e_1-1}}$, which is non-zero and thus becomes the leading term of the affine equation in standard form for $\Psi_{\eta}$. Thus, the ramification jump above $x=0$ is $e_1-1$.

Hence, the cover $\Psi_{\eta}$ is branched at three points that specialize to $0$ which have ramification jumps $e_1-1$ ,$e_2-1$ and $e_3-1$. By Lemma \ref{lemmagenus} and Lemma IV.2.3 of \cite{MR1011987}, $\Psi$ is a deformation of $\psi_0$ over $S$.

\end{proof}

Theorem \ref{deformationtheorem} $(b)$ is proved.

\subsection{Edge of Type $\{e\} \rightarrow \{p-1,e_1,e_2\}$} 
 
\begin{proposition}
\label{proptype3}
Suppose $\psi_0$ is an Artin-Schreier cover over $k$, branched at a point $b$ with ramification jump $e-1=p+e_1+e_2-2$ where $e_1, e_2, e_1+e_2-1 \not\equiv 1 \pmod p$. Then there exists an Artin-Schreier cover $\Psi$ over $S=\spec k[[t]]$ whose special fiber is isomorphic to $\psi_0$, whose generic fiber is branched at three points which specialize to b and which have ramification jumps $p-2$, $e_1-1$ and $e_2-1$, and whose ramification divisor is otherwise constant.

\end{proposition}

\begin{proof}

We define an Artin-Schreier cover $\Psi$ over $S=\spec k[[t]]$ given by the normalization of $\mathbb{P}^1_S$ over the extension of its fraction field that is defined by the following affine equation:

\[y^p-y=\dfrac{r-tx}{\frac{1-e_2}{e_1-1}x^p(x-t)^{e_1-1}(x-tr)^{e_2-1}}=:F(x,t)\]

\noindent Where $r=\frac{t^2-e_2+1}{e_1-1}$. Going through a process similar to one in the proof of Proposition \ref{proptype2}, one can show that $\Psi$ is a deformation satisfies the proposition. In particular, its special fiber is an Artin-Schreier cover over $k$, branched at $0$ with ramification jump $e-1$. Its generic fiber is an Artin-Schreier cover over $k((t))$ branched at $0, t$, and $tr$, which all specialize to $0$ and which have ramification jumps $p-2, e_1-1$, and $e_2-1$ respectively.

\end{proof}

Part $(c)$ of Theorem \ref{deformationtheorem} is proved.

\subsection{Edge of Type $\{(n+1)(p-n+1)\} \rightarrow \{n+1, n+1, \ldots,n+1\}$}
\label{deformation4}
 
 \begin{proposition}
\label{proptype4}

Suppose $\psi_0$ is an Artin-Schreier cover over $k$, branched at a point $b$ with ramification jump $(n+1)p-n^2$ where $p \nmid n[(n+1)p-n^2]$,  then there exists an Artin-Schreier cover $\Psi$ over $S=\spec k[[t]]$ whose special fiber is isomorphic to $\psi_0$, whose generic fiber is branched at $p-n+1$ points that specialize to b and which all have ramification jumps $n$, and whose ramification divisor is otherwise constant.

\end{proposition}

\begin{proof}

Let us study the Artin-Schreier cover $\Psi$ over $S=\spec k[[t]]$ given by the normalization of $\mathbb{P}^1_S$ over the extension of its fraction field that is defined by the following affine equation:

\[y^p-y=\dfrac{1}{x^p(x^{p-n}-t^{p-n})^n}=:F(x,t).\]

\noindent We will show that $\Psi$ is locally a deformation of $\psi_0$ that satisfies the conditions we seek.

On the special fiber, when $t=0$, we have:

\[F(x,0)=\dfrac{1}{x^{p+(p-n)n}}=\dfrac{1}{x^{(n+1)p-n^2}}.\]

\noindent Thus, the special fiber $\Psi_s$ has ramification jump $e-1$ at its only branch point. Hence, the cover $\Psi_s$ is birationally equivalent to $\psi_0$.

On the generic fiber, where $t \ne 0$, the derivative of $x^{p-n}-t^{p-n}$ (with respect to $x$) is $(p-n)x^{p-n-1}$, which differ from $0$ and only has $x=0$ as root. Thus $x^{p-n}-t^{p-n}$ is relatively prime to its derivative i.e. $x^{p-n}-t^{p-n}$ has $p-n$ distinct roots. Whence $F(x,t)$ on generic fiber has a pole at $x=0$, and $p-n$ distinct poles that are roots of  $x^{p-n}-t^{p-n}$. By a similar argument as before, $F(x,t)$ has poles of order exactly $n$ at those later points.

Let us consider the Taylor expansion of $F(x,t)$ around $x=0$:

\[\dfrac{(-1)^n}{t^{n(p-n)}x^p}(1+n(\dfrac{x}{t})^{p-n}+\cdots) \]

\noindent After a finite inseparable extension of $k((t))$ with equation $t_1^p=t$, the leading term of expansion of $F(x,t)$ around $0$ is a $p$-th power. The second term is $\frac{n (-1)^n}{t^{p(n+1)-n^2-n}x^n}$ which is nonzero and thus becomes the leading term of the affine equation in standard form. Thus, the ramification jump above $x=0$ is $n$.

Thus, the cover $\Psi_{\eta}$ is branched at $n+1$ points that specialize to $0$ which have ramification jumps all equal to $n$. By Lemma \ref{lemmagenus} and Lemma IV.2.3 of \cite{MR1011987}, $\Psi$ is a deformation of $\psi_0$ over $S$.

\end{proof}
 
 That proves part $(d)$ of Theorem \ref{deformationtheorem}.

 \subsection{Deformations in Characteristic $5$}
 
 In this section, the base field $k$ has characteristic $5$.
 
 \begin{proposition}
\label{deform3222}

Suppose $\psi_0$ is an Artin-Schreier cover over $k$, branched at a point $b$ with ramification jump $8$. Then there exists an Artin-Schreier cover $\psi_S$ over $S=\spec k[[t]]$ whose special fiber is isomorphic to $\psi_0$, whose generic fiber is branched at four points that specialize to b and which have ramification jumps $2$,$1$,$1$ and $1$, and whose ramification divisor is otherwise constant.

\end{proposition}

\begin{proof}

Let us study the Artin-Schreier cover $\Psi$ over $S=\spec k[[t]]$ given by the normalization of $\mathbb{P}^1_S$ over the extension of its fraction field that is defined by the following affine equation:

\[y^5-y=\frac{-rs+tx-tx^2}{-r_0s_0x^5(x-t)(x-tr)(x-ts)}=:F(x,t).\]

\noindent Where $r$, $s$ are solutions in $\overline{k((t))}$ of the equation:

\[X^2-(t^3-1)X+t^2-t^3+1=0.\]

\noindent And $r_0$, $s_0 \in k$ are the values of $r$ and $s$, respectively, when we plug in $t=0$. We will show that $\Psi$ is locally a deformation of $\psi_0$ that satisfies the conditions we seek.

On the special fiber, when $t=0$, $r_0$, $s_0$ are solutions in $k$ of the equation $X^2+X+1=0$ which does not have $0$ as a root. Thus $r_0$, $s_0 \neq 0$. Also, $X^2+X+1$ is separable on $k(X)$. Hence, $r_0 \neq s_0$. It follows that $r \neq s$, and it is also clear that $r,s \neq 1$. Moreover, the special fiber of $\Psi$ is defined by:

\[y^5-y=F(x,0)=\dfrac{-r_0s_0}{-r_0s_0x^8}=\dfrac{1}{x^8}.\]

\noindent Thus, the normalization of the special fiber $\Psi_s$ has ramification jump $8$ at its only branch point. Hence, the cover $\Psi_s$ is birationally equivalent to $\psi_0$.

On the generic fiber, $\Psi_{\eta}$ is branched above $x=0$, $x=t$, $x=tr$ and $x=ts$. The order of the poles of $F(x,t)$ at $x=t$, $x=tr$ and $x=ts$ is $1$. 

The expansion of $F(x,t)$ around $x=0$ is

\[ \dfrac{a_5}{x^5} +\sum_{i=1}^4 \dfrac{a_i}{x^i}+ \textrm{terms with non negative $x$-degree}, \]


\noindent where $a_5=-\frac{1}{r_0s_0t^3}$. In order for $\frac{a_{5}}{x^{5}}$ to have the same denominator as $F(x,t)$, we multiply the top and the bottom of the fraction by $\frac{(x-t)(x-t\cdot r)(x-t \cdot s)}{t^3}$. It would have the form:

\begin{equation}
\label{newtype51a5}
    \frac{(x-t)(x-tr)(x-ts)t^{-3}}{-r_0s_0x^5(x-t)(x-tr)(x-ts)}.
\end{equation}

\noindent One can check that the coefficients of the terms of the numerator of the above fraction with $x$-degree $0,1,2$ are precisely the corresponding ones of the numerator of $F(x,t)$. Thus, we have $a_{4}=a_3=0$. Moreover, we must have $a_{2}$ different from $0$, or else we can not cancel out the term with $x$-degree $3$ of the numerator of (\ref{newtype51a5}). After a finite inseparable extension of $k((t))$ with equation $t_1 ^5 =t$, the leading term $\frac{a_{5}}{x^{5}}$ of the expansion of $F(x,t)$ around $x=0$ is a $5$-th power. The second term is $\frac{a_{2}}{x^{2}}$, which is non-zero and thus becomes the leading term of the affine equation in standard form for $\Psi_{\eta}$. Thus, the ramification jump above $x=0$ is $2$.

Thus, the cover $\Psi_{\eta}$ is branched at four points that specialize to $0$ which have ramification jumps $2$ ,$1$, $1$ and $1$. By Lemma \ref{lemmagenus} and Lemma IV.2.3 of \cite{MR1011987}, $\Psi$ is a deformation of $\psi_0$ over $S$.

\end{proof}

 \begin{proposition}
\label{deform3322}

Suppose $\psi_0$ is an Artin-Schreier cover over $k$, branched at a point $b$ with ramification jump $9$. Then there exists an Artin-Schreier cover $\psi_S$ over $S=\spec k[[t]]$ whose special fiber is isomorphic to $\psi_0$, whose generic fiber is branched at four points that specialize to b and which have ramification jumps $2$,$2$,$1$ and $1$, and whose ramification divisor is otherwise constant.

\end{proposition}

\begin{proof}
Let us study the Artin-Schreier cover $\Psi$ over $S=\spec k[[t]]$ given by the normalization of $\mathbb{P}^1_S$ over the extension of its fraction field that is defined by the following affine equation:

\[y^5-y=\frac{rs-tx+tx^2}{r_0s_0x^5(x-t)^2(x-tr)(x-ts)}=:F(x,t).\]

\noindent Where $r$, $s$ are solutions in $\overline{k((t))}$ of the equation:

\[X^2-2(2t^3-t^2-2)X-2(t^3-2t^2-1)=0.\]

\noindent And $r_0$, $s_0$ are the values of $r$ and $s$, respectively, when $t=0$. We will show that $\Psi$ is locally a deformation of $\psi_0$ that satisfies the conditions we seek.

On the special fiber, when $t=0$, $r_0$, $s_0$ are solutions in $k$ of the equation $X^2+4X+2=0$ which does not have $0$ as a root. Thus $r_0$, $s_0 \neq 0$. Also, $X^2+4X+2$ is separable on $k(X)$. Hence, $r_0 \neq s_0$. It follows that $r \neq s \neq 1$.

\[F(x,0)=\dfrac{r_0s_0}{r_0s_0x^9}=\dfrac{1}{x^9}.\]

\noindent Thus, the normalization of the special fiber $\Psi_s$ has ramification jump $9$ at its only branch point. Hence, the cover $\Psi_s$ is birationally equivalent to $\psi_0$.

On the generic fiber, when $t \neq 0$, then $\Psi_{\eta}$ is branched above $x=0$, $x=t$, $x=tr$ and $x=ts$. 
The order of the poles of $F(x,t)$ at $x=t$, $x=tr$ and $x=ts$ is $1$. Expansion of $F(x,t)$ around $0$:

\[ \dfrac{a_5}{x^5} +\sum_{i=1}^4 \dfrac{a_i}{x^i}+ \textrm{terms with non negative $x$-degree,} \]


\noindent where $a_5=\frac{1}{r_0s_0t^4}$. In order for $\frac{a_{5}}{x^{5}}$ to have the same denominator as $F(x,t)$, we multiply the top and the bottom of the fraction by $\frac{(x-t)^2(x-t\cdot r)(x-t \cdot s)}{t^4}$. It would have the form:

\begin{equation}
\label{newtype52a5}
    \frac{(x-t)^2(x-tr)(x-ts)t^{-4}}{r_0s_0x^5(x-t)^2(x-tr)(x-ts)}.
\end{equation}

\noindent One can check that the coefficients of the terms of the numerator of the above fraction with $x$-degree $0,1,2$ are precisely the corresponding ones of the numerator of $F(x,t)$. Thus, we have $a_{4}=a_3=0$. Moreover, we must have $a_{2}$ different from $0$, or else we can not cancel out the term with $x$-degree $3$ of the numerator of (\ref{newtype52a5}). After a finite inseparable extension of $k((t))$ with equation $t_1 ^5 =t$, the leading term $\frac{a_{5}}{x^{5}}$ of the expansion of $F(x,t)$ around $x=0$ is a $5$-th power. The second term is $\frac{a_{2}}{x^{2}}$, which is non-zero and thus becomes the leading term of the affine equation in standard form for $\Psi_{\eta}$. Thus, the ramification jump above $x=0$ is $2$.

Thus, the cover $\Psi_{\eta}$ is branched at four points that specialize to $0$ which have ramification jumps $2$ ,$2$, $1$ and $1$. By Lemma \ref{lemmagenus} and Lemma IV.2.3 of \cite{MR1011987}, $\Psi$ is a deformation of $\psi_0$ over $S$.

\end{proof}

 We have proved part $(e)$ and completed the proof of Theorem \ref{deformationtheorem}.

\bibliographystyle{alpha}
\bibliography{cms4}

\begin{thebibliography}{Mau06}

\bibitem[BM00]{MR1767273}
Jos\'e Bertin and Ariane M\'ezard.
\newblock D\'eformations formelles des rev\^etements sauvagement ramifi\'es de
  courbes alg\'ebriques.
\newblock {\em Invent. Math.}, 141(1):195--238, 2000.

\bibitem[Cre84]{MR742696}
Richard~M. Crew.
\newblock Etale {$p$}-covers in characteristic {$p$}.
\newblock {\em Compositio Math.}, 52(1):31--45, 1984.

\bibitem[Eis95]{MR1322960}
David Eisenbud.
\newblock {\em Commutative algebra}, volume 150 of {\em Graduate Texts in
  Mathematics}.
\newblock Springer-Verlag, New York, 1995.
\newblock With a view toward algebraic geometry.

\bibitem[Har77]{MR0463157}
Robin Hartshorne.
\newblock {\em Algebraic geometry}.
\newblock Springer-Verlag, New York-Heidelberg, 1977.
\newblock Graduate Texts in Mathematics, No. 52.

\bibitem[Lan02]{MR1878556}
Serge Lang.
\newblock {\em Algebra}, volume 211 of {\em Graduate Texts in Mathematics}.
\newblock Springer-Verlag, New York, third edition, 2002.

\bibitem[M\'98]{MEZAR1998}
Ariane M\'{e}zard.
\newblock {\em Quelques problemes de deformations en caracteristique mixte}.
\newblock PhD thesis, 1998.
\newblock Th\`{e}se de doctorat Sciences et techniques communes Grenoble 1
  1998.

\bibitem[Mau06]{MR2223481}
Sylvain Maugeais.
\newblock Quelques r\'{e}sultats sur les d\'{e}formations \'{e}quivariantes des
  courbes stables.
\newblock {\em Manuscripta Math.}, 120(1):53--82, 2006.

\bibitem[Neu99]{MR1697859}
J\"urgen Neukirch.
\newblock {\em Algebraic number theory}, volume 322 of {\em Grundlehren der
  Mathematischen Wissenschaften [Fundamental Principles of Mathematical
  Sciences]}.
\newblock Springer-Verlag, Berlin, 1999.
\newblock Translated from the 1992 German original and with a note by Norbert
  Schappacher, With a foreword by G. Harder.

\bibitem[Pri03]{MR2016596}
Rachel~J. Pries.
\newblock Conductors of wildly ramified covers. {III}.
\newblock {\em Pacific J. Math.}, 211(1):163--182, 2003.

\bibitem[PZ12]{MR2985514}
Rachel Pries and Hui~June Zhu.
\newblock The {$p$}-rank stratification of {A}rtin-{S}chreier curves.
\newblock {\em Ann. Inst. Fourier (Grenoble)}, 62(2):707--726, 2012.

\bibitem[Ser79]{MR554237}
Jean-Pierre Serre.
\newblock {\em Local fields}, volume~67 of {\em Graduate Texts in Mathematics}.
\newblock Springer-Verlag, New York-Berlin, 1979.
\newblock Translated from the French by Marvin Jay Greenberg.

\bibitem[Sil86]{MR817210}
Joseph~H. Silverman.
\newblock {\em The arithmetic of elliptic curves}, volume 106 of {\em Graduate
  Texts in Mathematics}.
\newblock Springer-Verlag, New York, 1986.

\bibitem[SOS89]{MR1011987}
T.~Sekiguchi, F.~Oort, and N.~Suwa.
\newblock On the deformation of {A}rtin-{S}chreier to {K}ummer.
\newblock {\em Ann. Sci. \'Ecole Norm. Sup. (4)}, 22(3):345--375, 1989.

\end{thebibliography}

\Addresses

\end{document}